\documentclass[12pt,leqno,twoside]{amsart}

\usepackage{xcolor}
\usepackage{enumerate}

\usepackage{amssymb}

\usepackage[english]{babel}
\usepackage{amssymb,amsthm,amsmath,eucal,mathrsfs}
\usepackage{bm}

\newcommand*{\QEDB}{\hfill\ensuremath{\square}}%

\usepackage{cite}

\usepackage{amsmath,verbatim}

\usepackage{amsthm}

\usepackage{amssymb}

\usepackage{amsfonts}

\usepackage{dsfont}

\usepackage{hyperref}

\newtheorem{theorem}{Theorem}[section]

\newtheorem{lemma}[theorem]{Lemma}

\newtheorem{proposition}[theorem]{Proposition}

\newtheorem{corollary}[theorem]{Corollary}

\newtheorem{Hypotheses}{Hypotheses}

\theoremstyle{definition}

\newtheorem*{Proof}{Proof}
\numberwithin{equation}{section}

\renewcommand{\div}{\mathrm{div}\,}    

\newcommand\restr[2]{{
  \left.\kern-\nulldelimiterspace 
  #1 
  \vphantom{\big|} 
  \right|_{#2} 
  }}

\usepackage{eucal}

\title[Foldy-Lax approximation]{The Foldy-Lax approximation is valid for nearly resonating frequencies}
\author[Alsenafi, Ghandriche, Sini]{Abdulaziz Alsenafi$^{*}$, Ahcene Ghandriche $^{**} $, Mourad Sini$^{\ddag} $}
\date{\today}
\subjclass[2010]{35R30, 35C20}
\keywords{Foldy-Lax approximation, Multiple scattering, Resonances.}

\thanks{$^*$ Department of Mathematics, Faculty of Science, Kuwait University, P.O. Box 5969, Safat 13060, Kuwait. Email: abdulaziz.alsenafi@ku.edu.kw}
\thanks{$^{**}$ RICAM, Austrian Academy of Sciences, Altenbergerstrasse 69, A-4040, Linz, Austria. Email: ahcene.ghandriche@ricam.oeaw.ac.at.}
\thanks{$^{\ddag}$ RICAM, Austrian Academy of Sciences, Altenbergerstrasse 69, A-4040, Linz, Austria. Email: mourad.sini@oeaw.ac.at. }

\allowdisplaybreaks
\begin{document}
\maketitle

\begin{abstract} 
 
 The waves (including acoustic, electromagnetic and elastic ones) propagating in the presence of a cluster of inhomogeneities undergo multiple interactions between them. When these inhomogeneities have sub-wavelength sizes, the dominating field due to the these multiple interactions is the Foldy-Lax field. This field models the interaction between the equivalent point-like scatterers, located at the centers of the small inhomogeneities, with scattering coefficients related to geometrical/material properties of each inhomogeneities, as polarization coefficients. One of the related questions left open for a long time is whether we can reconstruct this Foldy-Lax field from the scattered field measured far away from the cluster of the small inhomogeneities. This is the Foldy-Lax approximation (or Foldy-Lax paradigm). In this work, we show that this approximation is indeed valid as soon as the inhomogeneities enjoy critical scales between their sizes and contrasts. These critical scales allow them to generate resonances which can be characterized and computed. The main result here is that exciting the cluster by incident frequencies which are close to the real parts of these resonances allow us to reconstruct the Fold-Lax field from the scattered waves collected far away from the cluster itself (as the farfields). In short, we show that the Foldy-Lax approximation is valid using nearly resonating incident frequencies. This results is demonstrated by using, as small inhomogeneities, dielectric nanoparticles for the $2$D TM model of electromagnetic waves and bubbles for the $3$D acoustic waves.             
\end{abstract}

\section{Introduction}

We are interested in the following model
\begin{equation}\label{sca-proble}
\begin{cases}
\nabla \cdot \left( a \nabla u \right) +k^2 b\; u & =0, \mbox{ in } \mathbb{R}^{m}, \,\, m=2 \,\, \text{or} \,\, 3, \\

u(x, \theta) & := u^i(x, \theta) +u^s(x, \theta)\\

\displaystyle\frac{\partial u^{s}}{\partial |x|}-i\; k \sqrt{\frac{b}{a}}\; u^{s}& =o\left(\displaystyle\frac{1}{|x|^{\frac{m-1}{2}}}\right), |x|\rightarrow\infty.
\end{cases}
\end{equation}
where 
\begin{equation}
a:=\begin{cases}
a_j \mbox{ in } D_j,\\
a_0 \mbox{ in } \mathbb{R}^{m} \setminus \underset{j=1}{\overset{M}{\cup}} D_{j}, 
\end{cases}
\quad \text{and} \quad
b:=\begin{cases}
b_j \mbox{ in } D_j, \\
b_0 \mbox{ in } \mathbb{R}^{m}\setminus \underset{j=1}{\overset{M}{\cup}} D_{j}, 
\end{cases}
\end{equation}
with constant values $a_j$ and $b_j$, $j=0, 1, ..., M$. Here $u^{i}$ is an incident wave, solution to $\Delta u^i+k^2\frac{b_0}{a_0}u^i=0$ in the whole space $\mathbb{R}^m$. Mostly, we deal with plane waves of the form $u^i:=u^i(x, \theta):=e^{i\; k \sqrt{\frac{b_0}{a_0}}\theta\cdot x}$ with an incident direction $\theta, \; \vert \theta\vert=1$. In dimension $m=2$, this model is related to the TM model in electromagnetism where $u(\cdot)$ is the third component of the electric field while $a_j$'s are the inverses of the magnetic permeability constants and $b_j$'s are the electric permittivity constants. We can also see $u(\cdot)$ as the third component of the magnetic field where, now, $b_j$'s are the inverses of the magnetic permeability constants and $a_j$'s are the electric permittivity constants. The inhomogeneities $D_j$'s model here nano-particles. In dimension $m=3$, this model is related to the propagation of acoustic wave $u(\cdot)$ in the presence of micro bubbles $D_j$'s. Here the constants $a_j$'s are the inverses of the mass density and the constants $b_j$'s are the inverses of the bulk moduli of the bubbles, \cite{Papanicoulaou-1, Papanicoulaou-2}. 
\bigskip

The inhomogeneities $D_j$'s are small scaled (i.e. micro scaled for the bubbles and nano scaled for the nano-particles). We take them of the form $D_j:=z_j +\delta\; B_j$, $j=1, ..., M$, where the $B_j$'s are Lipschitz smooth and bounded domains containing the origin (of $\mathbb{R}^m$), $z_j$'s model their location and $\delta$ their relative radius which we take to be small, i.e. $\delta=o\left( \underset{j=1,\cdots, M}{\max}\{\Vert B_j \Vert\} \right)$. We set $D:= \overset{M}{\underset{j=1}{\cup}} D_{j}$. 
Let $u(\cdot) = u^s(\cdot) + u^i(\cdot)$ be the solution of the acoustic scattering problem (\ref{sca-proble}).
 From the Lippmann-Schwinger representation of the total acoustic field $u(\cdot)$, we have
\begin{equation}
\label{eq:1}
 u(x) - \alpha \; \underset{x}{\div}  \int_D \Phi_{k}(x,y) \nabla u(y) dy -\beta k^2 \int_D \Phi_{k}(x,y)  u(y) dy = u^i(x),
\end{equation}
where $\alpha := a_1 - a_0$ and $\beta := b_1 - b_0$ represent the contrasts  between the inner and the outer acoustic coefficients. Here, $\Phi_{k}(\cdot,\cdot)$ is the Green's function of the background medium $(a_0, b_0)$ satisfying the outgoing Sommerfeld radiation conditions at infinity. This is an integro-differential equation. To transform it to a solely integral equation, we proceed by integration by parts, then \eqref{eq:1} becomes: 
\begin{equation*}
\label{eq:ReworkLippmannSchwinger}
  u(x) -  \gamma k^2  \int_D \Phi_{k}(x,y)u(y) dy + \alpha \int_{\partial D} \Phi_{k}(x,y) \frac{\partial u}{\partial \nu} (y) dy = u^i(x) ,
\end{equation*}
where $\gamma := \beta - \alpha b_1/a_1$, for $x \in D$. In addition, taking the normal derivative and trace, with the usual traces of the double layer potential, we obtain:
\begin{equation*}
\label{eq:ReworkLippmannSchwinger2}
\left( 1 + \frac{\alpha}{2} \right) \frac{\partial u}{\partial \nu} - \gamma k^2   \partial_{\nu-} \int_D \Phi_{k}(x,y)u(y) dy + \alpha( K_D^{k})^* \left[ \frac{\partial u}{\partial \nu} \right] = \frac{\partial u^i}{\partial \nu}.
\end{equation*}

Hence for $x\in \mathbb{R}^3\setminus \overline{D}$, the total acoustic field $u(x)$ is characterized by $u_{|_D} $ and $\frac{\partial u}{\partial \nu}\big|_{\partial D}$,  which are solutions of the following close form system of integral equations:
\bigskip
\begin{eqnarray}
\left[ I -  \gamma k^2 A^k \right] u + \alpha \int_{\partial D} \Phi_{k}(x,y) \frac{\partial u}{\partial \nu} (y) dy &=& u^i(x), \; \mbox{in} \,\, D,  \label{eq:ReworkLippmannSchwinger-1-} \\
\left[ \frac{1}{\alpha} +\frac{1}{2} + ( K_D^{k})^{*} \right] \left[ \frac{\partial u}{\partial \nu} \right] - \frac{\gamma}{\alpha} k^2   \partial_{\nu-} \int_D \Phi_{k}(x,y)u(y) dy & = & \frac{1}{\alpha}\frac{\partial u^i}{\partial \nu}, \; \mbox{on} \,\, \partial D, 
\label{eq:ReworkLippmannSchwinger2-}
\end{eqnarray} 

with the Newtonian (a volume-type) operator:
$$
A^{k} : \mathbb{L}^2(D)\longrightarrow \mathbb{L}^2(D),~~~ A^k(u)(x):=\int_D \Phi_{k}(x,y)u(y) dy,
$$
with image of $A^{k}$ in $\mathbb{H}^2(D)$, and the Neumann-Poincar\'e (a surface-type) operator\footnote{The notation $p.v$ means the Cauchy principal value.}
$$
(K_D^{k})^*: \mathbb{H}^{-1/2}(\partial D) \longrightarrow \mathbb{H}^{-1/2}(\partial D),~~~ (K_D^{k})^*(f)(x):=p.v.\int_{\partial D} \frac{\partial }{\partial \nu(x)} \Phi_{k}(x,y) f (y) d\sigma(y).
$$
We recall that for $k=0$, each of these operators generates a sequence of eigenvalues: $\lambda_m(A^0)\stackrel{m\rightarrow \infty}{\longrightarrow} 0$ and $\sigma_p((K_D^{0})^*)\subset [ -\frac{1}{2}, \frac{1}{2})$. In addition, we have $K^0_D(1)= - \, \frac{1}{2}$. These singular values are behind the resonances that we want to use. 
\bigskip

We consider the coefficients $a$ and $b$ satisfying the following scales in terms of $\delta$. We provide the following situations where the constants $a_0$ and $b_0$ in the background are assumed to be independent (or uniform in terms of) of $\delta$.

\begin{enumerate}
\item[]
\item First situation. 
\begin{equation}
\label{First}
a:= \begin{cases}
a_j\sim \delta^{-2} \quad \text{inside } D_j,\\
a_0 \quad \text{outside } \overset{M}{\underset{j=1}{\cup}} D_j , \end{cases}
\text{and} \qquad
b:= \begin{cases}
b_j \sim \delta^{-2} \quad \text{inside } D_j,\\
b_0 \quad \text{outside } \overset{M}{\underset{j=1}{\cup}} D_j. \end{cases}\end{equation}
In this case, we have $\gamma \sim 1$ and then $\gamma \, k^2 \, A^{k} \ll1$ as $a\ll1$. Hence, there is no singularity coming from (\ref{eq:ReworkLippmannSchwinger-1-}). But as $\alpha \gg1$, precisely if $\alpha \sim a^{-2}$ as $a\ll1$, then we can excite the eigenvalue $-\frac{1}{2}$ of $K_D^{0}$ and create a singularity in (\ref{eq:ReworkLippmannSchwinger2-}). In this case, we have the \emph{Minnaert resonance} with surface-modes.
Its dominant part has the following value 
\begin{equation*}
k^2_M:=\sqrt{\frac{8 \, \pi \, a_0}{a_1 \, \Theta_{\partial D}}} \,\,\, \text{with} \,\,\, \Theta_{\partial D}:= \frac{1}{\left\vert \partial D \right\vert} \int_{\partial D}  \int_{\partial D} \frac{(x-y) \cdot \nu(x)}{4 \pi |x-y| } \, d\sigma(x) d\sigma(y) = \delta^2\; \Theta_{\partial B}.
\end{equation*}
This resonance was first observed in \cite{Habib-Minnaert} based on indirect integral equation methods. This observation was used for different purposes, see \cite{HFGLZ1, A-F-G-L-Z, A-F-L-Y-Z}.
This result was extended to more general families of micro-bubbles in  \cite{ACCS-JDE, ACCS-effective-media}.   

\bigskip

\item Second situation.\\
Observe that if $\alpha$ is negative as in the case of the Drude model for electromagnetism (recall here that $a$ plays the role of the inverse of the permeability and $b$ the one of the permittivity) i.e. 
\begin{equation}\label{Second}
\epsilon:=\epsilon_0 -\frac{k_p^{2}}{k(k +i\; \gamma_{dp})}~~ \left( \text{or} \; \mu:=\mu_0 -\frac{k_p^{2}}{k(k +i\; \gamma_{dp})} \right),
\end{equation}
with $k_p$ is the plasma frequency, $\gamma_{dp}$ the damping parameter and $k$ is our incident frequency. Then we could excite the other sequence of eigenvalues, $\sigma_n$, of $(K_D^{0})^*$ by choosing $\gamma_{dp}$ small in terms of $\delta$ and $k$ near the plasma frequency. Indeed, the operator $\displaystyle\frac{1}{\alpha} +\frac{1}{2} + ( K_D^0)^*$ is non injective if $k$ solve the dispersion equations $\displaystyle\frac{1}{\alpha}+\frac{1}{2}+ \sigma_p((K_D^{0})^*)=0$ which we write in terms of the Drude model as $\displaystyle\frac{\epsilon_{0} \, \epsilon}{\epsilon_{0} - \epsilon}+\frac{1}{2}+ \sigma_p((K_D^{0})^*)=0$ which gives, using $(\ref{Second})$ the dispersion equations $\displaystyle\frac{\epsilon_{0} \left[ k (k+ i \, \gamma_{dp}) \epsilon_{0} - k_{p}^{2} \right]}{k^{2}_{p}} +\frac{1}{2}+\sigma_n=0$. As $\frac{1}{2}+\sigma_n\geq 0$, then we should take the frequency $k^{2} < \displaystyle\frac{k_{p}^{2}}{\epsilon_{0}}$ so that real part of $\epsilon$ is negative and hence we can have solution of the above dispersion equations. This gives us another sequence of resonances $k^{2}_{n} := \displaystyle\frac{k^{2}_{p}}{\epsilon_{0}} \left(\epsilon_{0}  - \frac{1}{2} - \sigma_{n} \right)$ (i.e. corresponding to the sequence of plasmonics in electromagnetism). 

\bigskip

\item Third situation
 \begin{equation}
\label{Third}
a:= \begin{cases}
a_j\sim 1 \quad \text{inside } D_j,\\
a_0 \quad \text{outside } \overset{M}{\underset{j=1}{\cup}} D_{j}, \end{cases}
\text{and} \qquad
b:= \begin{cases}
b_j \sim \delta^{-2} \quad \text{inside } D_j,\\
b_0 \quad \text{outside } \overset{M}{\underset{j=1}{\cup}} D_j. \end{cases}\end{equation}

\bigskip

In this case, we have $\alpha \sim 1$ and then we keep away from the full spectrum of $(K_D^{0})^*$. Hence, there is no singularity coming from (\ref{eq:ReworkLippmannSchwinger2-}). But as $\gamma \sim \delta^{-2}\gg1$, we can excite the eigenvalues of the Newtonian operators $A^0$ and create singularities in (\ref{eq:ReworkLippmannSchwinger-1-}). This gives us a sequence of resonances with volumetric-modes. 
Indeed, we can write, after scaling,  
\begin{equation*}
I -  \gamma k^2 A^0 = I -  \gamma k^2 \,\delta^2 \, \widetilde{A^{0}},
\end{equation*}
where 
\begin{equation*}
\widetilde{A^{0}} \; v(x) := \int_{B} \, \Phi_0(\vert x-y\vert) \, v(y) \, dy.
\end{equation*}
Let $\{ \tilde{\lambda}_{n} \}_{n \geq 1},$ be the sequence of eigenvalues of $\widetilde{A^{0}}$. Then, the operator $[I -  \gamma k^2 A^0]$ is not injective if $k$ satisfies one of the following dispersion equations $1-\gamma k^2 \delta^2 \tilde{\lambda}_n=0$. The values $k_n:=\sqrt{\displaystyle\frac{1}{\gamma \, \delta^2 \, \tilde{\lambda}_n}}$ are called the dielectric resonances. This was observed in \cite{A-D-F-M-S} and \cite{M-M-S}. 

\end{enumerate}

The scattered field by the collection of inhomogeneities $D_j, j=1, ..., M,$ is given by a linear combination of poles or dipoles centered at the centers of the inhomogeneities. We call this dominating field the Foldy-Lax field as it is of the form of the scattered field generated by point-like scatterers (or Dirac-like potentials).
This field encodes the multiples scattering effect between the small inhomogeneities. 
 A natural, and actually an outstanding open question, was whether from the scattered field measured away from the collection of inhomogeneities we can reconstruct the field due to multiple interaction between them. 
In mathematical terms, is the approximating term, i.e. the Foldy-Lax field or its iterations (see later), dominating the error of the approximation? We call this property, the Foldy-Lax paradigm (or approximation). 
 There are attempts to justify this paradigm but without real success. 
To our knowledge, the only published paper addressing this issue is \cite{CASSIER} where acoustic waves in the harmonic regime was considered with sound-soft inhomogeneities (obstacles with Dirichlet condition). 
The Foldy-Lax approximation was justified when the $D_j$'s are discs. The particular geometry of these inhomogeneities allow them to perform exact computations. 
\bigskip

In this work, we show that for general shapes and even general transmission condition the Foldy-Lax approximation is valid as soon as nearly resonating incident frequencies are used. Our message is that the main reason why we can see the multiple scattering between the inhomogeneities is the fact that we excite them with incident frequencies close to their own resonances. To justify this claim, we consider inhomogeneities satisfying the conditions in (\ref{Third}), assuming that $a_{j}=a_{0}$ as well, and take $k$ close to the dielectric resonances $k_{n} := \sqrt{1 / \gamma \, \delta^{2} \, \lambda_{n}}$.
However, we do believe that similar results can be shown using the other condition (\ref{First}) and (\ref{Second}) and use incident frequencies close to the Minnaert resonance $k_M:=\sqrt{\displaystyle\frac{8 \, \pi \, a_0}{a_1 \, \Theta_{\partial D}}}$ or the plasmonic resonances $\displaystyle\frac{k^{2}_{p}}{\epsilon_{0}} \left(\epsilon_{0}  - \frac{1}{2} - \sigma_{n} \right)$.   
\bigskip
\newline

To state mathematically these results, we first set $d:=\underset{i\neq j \atop i,j=1; \cdots ; M}{\min}\; dist(D_i, D_j)$ and assume this minimum distance to be of the order:
\begin{equation}\label{d03D}
d=d_0\; \delta^t,\; t\geq 0, 
\end{equation} 
in the $3$ dimension case and
\begin{equation}\label{d02D}
d = d_0\; e^{-\vert log(\delta)\vert^{t}},\; t\geq 0,
\end{equation}
in the $2$ dimension case. In both cases, the constant  $d_0>0$ is independent on  $\delta$.

\bigskip

In the next theorem, we discuss the Foldy-Lax approximation when the inhomogeneities satisfy (\ref{Third}) in the $3$ dimension case. 
In this case, we recall that the Newtonian operator $A^0:=A^0_j$, on any $D_j$, reads
\begin{equation}\label{Newtonian-potential-3D}
A^0 : \mathbb{L}^2(D_j)\longrightarrow \mathbb{L}^2(D_j),~~~ A^0(u)(x):=\int_{D_j} \Phi_{0}(x,y)u(y) dy=\int_{D_j}\frac{1}{4\pi \vert x-y\vert}\; u(y) dy.
\end{equation}

\begin{theorem}\label{3D-case} Let $k_n:=\sqrt{\displaystyle\frac{1}{\gamma \, \lambda_n}}$ where $\lambda_n$ is any eigenvalue of the Newtonian operator $A^0$ \footnote{Recall that in 3D, $\lambda_n=\delta^2 \tilde{\lambda}_n$ where $\tilde{\lambda}_n$ are the eignevalues of $\widetilde{A_0}$.}. We choose the incident frequency $k$ satisfying
\begin{equation}\label{omega}
k^2:=k^{2}_{n}\left( 1 \pm \delta^{h} \right),\; \delta <<1.
\end{equation}
Under the condition $t+h\leq 1$, we have the expansion:
\begin{equation}\label{intro-usca}
u^{s}(x,\theta)  =  \sum^M_{j=1} \Phi_{k}(x,z_j) \; C_{j} \; Q_j \; + \mathcal{O}\left( \delta^{\min\{2-h, 3-2h-2t\}} \right),
\end{equation}
where $Q:=(Q_1, ..., Q_M)$ is the unique solution of the algebraic system:
\begin{equation*}
\left( I - B_{k} \right) \cdot Q \; = \; U,
\end{equation*}  
where $U:=(u^i(z_1, \theta), \cdots , u^i(z_M, \theta))$ and $B_k:=\left( B_{k, i, j} \right)^{M}_{i, j=1}$ with\footnote{Where $\bm{\delta}$ is the Kronecker symbol.} 
\begin{eqnarray*}
B_{k, i, j} &:=& \frac{C_j}{\left( 1 - i \, k \, C_{i} \right)} \,\, \Phi_{k}(z_i, z_j) \, \left(1 - \underset{i,j}{\bm{\delta}} \right) \\ & \mbox{and} & \\  C_{j} &:=& \int_{D_j} \left[ \frac{a_{0}}{k^2 \, b}I - A^{0} \right]^{-1}(1)(x) \, dx.
\end{eqnarray*}
\end{theorem}
As a consequence of the previous theorem, we state the following corollary.
\begin{corollary}\label{coro12}
For any integer $N$, we define:
\begin{equation*}
Q^{N} := \sum_{n=0}^{N} \, B_{k}^{n} \cdot U \quad \text{and} \quad u^{s,N}(x,\theta)  :=  \sum^M_{j=1} \Phi_{k}(x,z_j) \; C_{j} \; Q^N_j.
\end{equation*}

\begin{enumerate}

\item Under the following condition:
\begin{equation}\label{intro-condition-Foldy-Lax-proof}
0\leq 1-h-t \leq \min\left\{ \frac{1}{N+1}, \frac{1-t}{N} \right\},
\end{equation}
we have 
\begin{equation}\label{intro-N-interactions-proof}
u^{s}(x,\theta) - u^{s,N}(x,\theta) =  \mathcal{O}\left( \delta^{(1-h) + (N+1)(1-t-h)} \right)
\end{equation}
and  
\begin{equation}\label{intro-N-interactions-scattered-field}
u^{s,N}(x,\theta) - u^{s,N-1}(x,\theta)\; \sim \;  \delta^{(1-h) + N(1-t-h)} \gg \delta^{(1-h) + (N+1)(1-t-h)}, 
\end{equation}
for any bounded domain away from the collection of centers $z_j, j=1, ..., M$. Here, $u^{s,N}(x,\theta)$ is the scattered field after $N$-interactions between the nano-particles.  
\bigskip

\item Under the limit condition $1-t-h=0$ and $C_{0} \, d_{0} < 1$, where $C_{0}$ is such that $C = C_{0} \, \delta^{1-h}$ and $d_{0}$ satisfy $(\ref{d03D})$, we have the expansion:
\begin{equation}\label{intro-infty-interactions-proof}
u^{s}(x,\theta) - u^{s,\infty}(x,\theta) = \mathcal{O}\left( \delta \right),\;  \delta \ll 1,
\end{equation}
where $$u^{s,\infty}(x,\theta):= \sum^M_{j=1} \Phi_{k}(x,z_j) \; C_{j} \; Q^\infty_j \sim \delta^{1-h} \sim \delta^t $$ for in any bounded domain away from the collection of centers $z_{j}, j=1, \cdots, M$. Here $u^{s,\infty}(x,\theta)$ is the field generated after all the interactions between the particles, i.e. $Q^\infty:=(Q_1^\infty, \cdots, Q_M^\infty)$ is given by 
\begin{equation}
 Q^\infty=\sum^\infty_{n=0}B^n_k \cdot U.
\end{equation} We call $u^{s,\infty}(\cdot,\theta)$ the Foldy-Lax field.
\end{enumerate}
\end{corollary}
\vspace{3mm}
The part (1) of Corollary \ref{coro12} means that if we want to reconstruct the field encoding the interactions between the particles until the order $N$, for a given integer $N$, we need to use an incident frequency close to the resonance at the order $O(a^h)$ with $h$ satisfying (\ref{intro-condition-Foldy-Lax-proof}). In addition, under the optimal condition $1-t-h=0$, we can see all the interactions between the particles. Therefore under this condition between the closeness of the particles and closeness of the incident frequency to the used chosen resonance, the Foldy field can be fully reconstructed and hence the whole interactions between the particles can be 'seen'.
\bigskip

The coming corollary, gives more details about the case when the particles are away from each other, i.e. when we can take $t=0$. 
\begin{corollary}\label{coro13}
In this case, the condition (\ref{intro-condition-Foldy-Lax-proof}), reads as:
\begin{equation}\label{intro-condition-Foldy-Lax-proof-t=0}
0 \leq 1-h \leq \frac{1}{N+1}.
\end{equation}
and  
\begin{equation}\label{intro-N-interactions-proof-t=0}
u^{s}(x,\theta) - u^{s,N}(x,\theta)=\mathcal{O}\left( \delta^{(N+2)(1-h)} \right),
\end{equation}
with 
\begin{equation*}
u^{s,N}(x,\theta) - u^{s,N-1}(x,\theta) \sim \delta^{(N+1)(1-h)} \gg  \delta^{(N+2)(1-h)}.
\end{equation*}

\end{corollary}
\bigskip

The estimate (\ref{intro-N-interactions-proof-t=0}) 
means that the Foldy-Lax approximation, at a given order of interaction  $N$, is also valid when the nano-particles are away from each other as soon as the used incident frequency is close to the resonance $k_{n_0}$, i.e. (\ref{omega}) is valid with $h\in (0, 1)$ such that
(\ref{intro-condition-Foldy-Lax-proof-t=0}) is satisfied.
\bigskip

Next, we state the corresponding results for the $2$ dimensional case.
 For this, we state the natural Newtonian operator $A^0:=A^0_j$ as:
\begin{equation}\label{Newtonian-potential-2D}
A^0 : \mathbb{L}^2(D_j) \longrightarrow \mathbb{L}^2(D_j),~~~ A^0(u)(x):=\int_{D_j} \Phi_{0}(x,y)u(y) dy=-\int_{D_j}\frac{1}{2\pi} \log(\vert x-y\vert) \; u(y) dy.
\end{equation}
In dimension 2, we need the following contrast of the coefficients $a_j$ and $b_j$.
\begin{equation}
\label{Third-2D}
a:= \begin{cases}
a_j\sim 1 \quad \text{inside } D_j,\\
a_0 \quad \text{outside } \overset{M}{\underset{j=1}{\cup}} D_{j}, \end{cases}
\text{and} \qquad
b:= \begin{cases}
b_j \sim \delta^{-2}\left\vert \log(\delta) \right\vert^{-1} \quad \text{inside } D_j,\\
b_0 \quad \text{outside } \overset{M}{\underset{j=1}{\cup}} D_j. \end{cases}\end{equation}
Compared to the 3D case, here we have contrasts of the order $\delta^{-2}\left\vert \log(\delta) \right\vert^{-1}$ instead of $\delta^{-2}$. Such scales are dictated by the $\log$-type singularity of the fundamental solution, see \cite{2DGHANDRICHE}. Therefore, we need the following hypothesis on the behavior of the eigenvalues and eigenfunctions of the 2D Newtonian operator:

\begin{Hypotheses}\label{hyp}
The particles $D_j$, of radius $\delta,\; \delta \ll 1$, are taken such that the spectral problem \, $A_0 u =\lambda\; u, \mbox{ in } D_j$, has eigenvalues $\lambda_n$ and corresponding eigenfunctions, $e_n$, satisfying the following properties:
\newline

\begin{enumerate}
 \item $\int_{D_j} e_n(x) dx \neq 0, \; \forall \delta \ll 1.$
 \bigskip
 
 \item $\lambda_n \sim \delta^2 \vert \log(\delta)\vert, \; \forall \delta \ll 1.$ 
\end{enumerate}
\end{Hypotheses}

In \cite{2DGHANDRICHE}, see the appendix there, it is shown that for particles of general shapes, the first eigenvalue and the corresponding eigenfunctions satisfy {\bf{Hypotheses}} \ref{hyp}. 
In addition, the properties of the eigenvalues for the case when the particles are discs are characterized.

\begin{theorem}\label{2D-case} Assume {\bf{Hypotheses}} \ref{hyp} to be satisfied and let $k_{n} := \sqrt{\displaystyle\frac{1}{\gamma  \; \lambda_n}}$ where $\lambda_n$ is any eigenvalue of the Logarithmic Newtonian operator $A^0$ \footnote{From (\ref{Third-2D}), we see that $\gamma \sim \delta^{-2} \vert \log(\delta)\vert^{-1}$. Therefore and the second part of {\bf{Hypotheses}} \ref{hyp}, we deduce that $k_n =O(1)$ for $\delta \ll 1$.}. We choose the incident frequency $k$ satisfying
\begin{equation}\label{omega-2D}
k^2:=k^2_{n} \left( 1 \pm \left\vert \log(\delta) \right\vert^{-h} \right),\;\; \delta<<1.
\end{equation}
Under the condition $t+h\leq 1$, we have the expansion:
\begin{eqnarray}\label{intro-usca-2D}
\nonumber
u^{s}(x,\theta) &=& \sum^M_{j=1} \Phi_{k}\left(x, z_{j} \right) \; C_{j}^{\star} \; Q_j +  \mathcal{O}\left( \delta^{1-t} \, \left\vert \log(\delta) \right\vert^{2 \, (h-1)} \right) + \mathcal{O}\left( \delta \, \left\vert \log(\delta) \right\vert^{(h-1)} \right),
\end{eqnarray}
where $Q:=\left( Q_{1},\cdots, Q_{M} \right)$ is the unique solution of the algebraic system
\begin{equation*}
\left( I - B_{k} \right) \cdot Q \; = \; U,
\end{equation*}  
where $U := \left( u^i(z_{1}, \theta), \cdots, u^i(z_{M}, \theta) \right)$ and $B_k:=(B_{k, i, j})^M_{i, j=1}$ where 
\begin{equation}
B_{k,i, j} := \Phi_{k} \left(z_{i}, z_{j}\right) \; \left[ C^{-1}_{j} \, - \, \bm{E} \right]^{-1} \; \left(1 - \underset{i,j}{\bm{\delta}} \right)
\end{equation}
and  
\begin{eqnarray*}
C_{j} &:=& \int_{D_j} \left[ \frac{a_{0}}{k^2 \,   b}I - A^{0} \right]^{-1}(1)(x)dx 
\end{eqnarray*}
with 
\begin{equation*}
\bm{E} := \frac{i}{4} - \frac{1}{2 \pi}\left[ \log\left( \frac{k}{2} \right) + \displaystyle\lim_{p \rightarrow +\infty} \left( \sum_{m=1}^{p} \frac{1}{m} - \log(p) \right) \right].
\end{equation*}
\end{theorem}
The previous theorem suggest the following corollary.
\begin{corollary}\label{coro15}
For any integer $N$, we define
\begin{equation*}
Q^{N} := \sum_{n=0}^{N} \, B_{k}^{n} \cdot U \quad \text{and} \quad u^{s,N}(x,\theta)  :=  \sum^M_{j=1} \Phi_{k}(x, z_{j}) \; \left[ C_{j}^{-1} \, - \, \bm{E} \right]^{-1} \; Q^N_j.
\end{equation*}

\begin{enumerate}

\item Under the following general condition, which is uniform in terms of $N < \infty$,
\begin{equation}\label{intro-condition-Foldy-Lax-proof-2D}
1 - t - h > 0
\end{equation}
 we have 

\begin{equation}\label{intro-N-interactions-proof-2D}
u^{s}(x,\theta) - u^{s,N}(x,\theta) =  \mathcal{O}\left( \left\vert \log(\delta) \right\vert^{(h-1) -(N+1)(1-t-h)} \right)
\end{equation}
and  
\begin{equation}\label{intro-N-interactions-scattered-field-2D}
u^{s,N}(x,\theta) - u^{s,N-1}(x,\theta) \sim \left\vert \log(\delta) \right\vert^{(h-1)-N(1-t-h)} \gg  \left\vert \log(\delta) \right\vert^{(h-1) -(N+1)(1-t-h)} 
\end{equation}
for in any bounded domain away from the collection of centers $z_{j}, j=1, \cdots, M$.
\bigskip

\item Under the limit condition $1-t-h=0$, we have the expansion
\begin{equation}\label{intro-infty-interactions-proof-2D}
u^{s}(x,\theta) - u^{s,\infty}(x,\theta) =  \mathcal{O}\left( \delta^{1-t}\; \left\vert \log(\delta)\right\vert^{-2 \, t} \right),\;  \delta \ll 1,
\end{equation}
where 
\begin{equation*}
u^{s,\infty}(x,\theta):= \sum^M_{j=1} \Phi_{k}(x, z_{j}) \; \left[ C_{j}^{-1} - \bm{E} \right]^{-1} \; Q^\infty_j \sim  \left\vert \log(\delta) \right\vert^{-t} 
\end{equation*}
for in any bounded domain away from the collection of centers $z_j, j=1, \cdots, M$. Here $u^{s,\infty}(x,\theta)$ is the field generated after all the interactions between the particles, i.e. $Q^\infty:=(Q_1^\infty, \cdots, Q_M^\infty)$ is given by 
\begin{equation*}
 Q^\infty=\sum^\infty_{n=0}B_{k}^{n} \cdot U.
\end{equation*} 
\end{enumerate}
\end{corollary}
We observe that, contrary to the $3$ dimensional case, the scattered field generated after $N$ interactions (for any order $N$), can be reconstructed from measured scattered field away from the cluster of particles. 
In addition, the whole Foldy-Lax field can be reconstructed as well. Observe also that, we can take $t=0$ (and then $h=1$), which means that the particles can be far away from each other and still we have the Foldy-Lax approximation (for $N\leq \infty$) if the incident frequency is close to the resonance with the (optimal value) $h=1$. This means that in the $2D$ case, the Foldy-Lax approximation is always valid. 
In the case $N< \infty$, we can also take $h=0$ in (\ref{intro-condition-Foldy-Lax-proof-2D}) and hence we do not need to use nearly resonating incident frequencies (this is not the case for $N=\infty$). As in (\ref{intro-condition-Foldy-Lax-proof-2D}), $h=0$ implies that $t=1$, therefore the Foldy-Lax approximation is valid but in the mesoscale regime, see (\ref{d02D}), i.e. $d \sim \delta$, if the incident frequency is not close to any resonance.
\newline
At the mathematical analysis level, in the $2$ D case, the singularity of the fundamental solution is of log-type and hence its powers stay always dominating the singularity of its derivatives. This is not true in the $3$ D case.  
\newline  
The results in Theorem \ref{2D-case} confirm, in particular, the ones in \cite{CASSIER} where in the 2D case, it is shown that the Foldy-Lax approximation is valid for sound-soft acoustic obstacles of discs shapes. 
\bigskip

Let us finish this introduction by mentioning that the Foldy-Lax approximations have different applications, see \cite{Martin:2006} for classical results and historical facts. This approximation is called after Foldy and Lax as the dominant part of these approximation is nothing but the field generated by point-like inhomogeneity as modeled by Foldy \cite{Foldy} and Lax \cite{Lax}. In addition, such approximations are useful in modern mathematical imaging see for instance \cite{Alsaedi_2016, A-A-C-K-S, 2DGHANDRICHE} and the effective medium theory 
\cite{CMS-2017,ACCS-effective-media}.

\section{Proof of Theorem \ref{3D-case}}
We start the proof by setting the Lippmann-Schwinger equation, for several particles, given for $x \in D_{m}$ by: 
\begin{equation*}
v_{m}(x) - k^2 \, \frac{1}{a_{0}} \, \tau \int_{D_{m}} \Phi_{k}(x,y) v_{m}(y) dy - k^{2} \, \frac{1}{a_{0}} \, \tau \sum_{j \neq m} \int_{D_{j}} \Phi_{k}(x,y) v_{j}(y) dy = u^{i}(x),
\end{equation*}
where we recall that\footnote{To avoid more complicated notations and simplify the exposition, we assume that $\tau_{j} := b - b_{0}(z_{j}) = \tau, \,\, j=1,\cdots,M.$ The final results stay valid without this assumption.} $\tau:= b - b_{0}(z) \sim \delta^{-2}$ and $v_{m}(\cdot)=u(\cdot)_{|_{D_{m}}}$. To make appear the Newtonian operator, we rewrite the previous equation as:
\begin{eqnarray*}
v_{m}(x) &-& k^2 \, \frac{1}{a_{0}} \, \tau \, \int_{D_{m}} \Phi_{0}(x,y) \, v_{m}(y) \, dy - \, k^{2} \, \frac{1}{a_{0}} \, \tau \, \sum_{j \neq m} \int_{D_{j}} \Phi_{k}(x,y)\, v_{j}(y)\, dy \\ &=& u^{i}(x)  + k^{2} \, \frac{1}{a_{0}} \, \tau \,  \int_{D_{m}} (\Phi_{k}-\Phi_{0})(x,y) \, v_{m}(y) \, dy.
\end{eqnarray*}
By expanding, using Taylor expansion, the functions $\Phi_{k}(\cdot,\cdot)$ and $u^{i}(\cdot)$ near the centers and using the series representation of $(\Phi_{k} - \Phi_{0})$, we obtain: 
\begin{eqnarray*}
\left[I - k^{2} \, \frac{1}{a_{0}} \, \tau \, A^{0} \right] v_{m}(x) &-&  \, k^{2} \, \frac{1}{a_{0}} \, \tau \, \sum_{j \neq m}   \Phi_{k}(z_{m},z_{j}) \, \int_{D_{j}} \, v_{j}(y)\, dy \\ &-& \, k^{2} \, \frac{1}{a_{0}} \, \tau \, \sum_{j \neq m} \int_{D_{j}} \,  \int_{0}^{1}\underset{x}{\nabla} \Phi_{k}(z_{m}+t(x-z_{m}),z_{j})\cdot(x - z_{m}) \, dt  \, v_{j}(y)\, dy \\ &-& \, k^{2} \, \frac{1}{a_{0}} \, \tau \, \sum_{j \neq m} \int_{D_{j}}  \int_{0}^{1} \underset{y}{\nabla} \Phi_{k}(x,z_{j}+t(y-z_{j})) \cdot (y-z_{j}) \, dt  \, v_{j}(y)\, dy \\ &=& u^{i}(z_{m}) + \int_{0}^{1} \nabla u^{i}(z_{m}+t(x-z_{m})) \, \cdot (x-z_{m}) \, dt \\ &+& k^{2} \, \frac{1}{a_{0}} \, \tau \,  \int_{D_{m}} \left[ i k +\sum_{\ell \geq 1} \frac{(i \, k)^{\ell + 1}}{(\ell + 1)!} \, \left\vert x-y \right\vert^{\ell} \right]  \, v_{m}(y) \, dy.
\end{eqnarray*}
We set the scattering coefficient $w$ to be: 
\begin{equation*}
w := k^{2} \, \frac{1}{a_{0}} \, \tau \, \left[ I - k^2 \, \frac{1}{a_{0}} \, \tau \, A^{0} \right]^{-1}(1).
\end{equation*}
On both sides of the previous equation, we take the inverse operator of $\left[ I - k^2 \, \frac{1}{a_{0}} \, \tau \, A^{0} \right]$ and then integrating over $D_{m}$ to obtain: 
\begin{eqnarray}\label{AlgebraicSystem}
\nonumber
\left[ 1 - i k \int_{D_{m}} w(x) \, dx \right] \, \int_{D_{m}} v_{m}(x) \, dx & = & \frac{a_{0}}{k^{2}  \, \tau} u^{i}(z_{m}) \int_{D_{m}} w(x) \, dx  \\ \nonumber
&+& \sum_{j \neq m} \Phi_{k}(z_{m},z_{j}) \int_{D_{j}} v_{j}(y) \, dy \, \int_{D_{m}} w(x) \, dx \\
&+& Remainder,
\end{eqnarray}
where the remainder term is defined by: 
\begin{eqnarray}\label{Remainder}
\nonumber
Remainder & := & \frac{a_{0}}{k^{2} \, \tau} \, \int_{D_{m}} w(x) \, \int_{0}^{1} \nabla u^{i}(z_{m}+t(x-z_{m})) \centerdot (x-z_{m}) \, dt \, dx \\ \nonumber
&+& \int_{D_{m}} w(x) \int_{D_{m}} \sum_{\ell \geq 1} \frac{(ik)^{\ell + 1}}{(\ell + 1)!} \, \left\vert x-y \right\vert^{\ell} v_{m}(y) \, dy \, dx \\ \nonumber
&+&\int_{D_{m}} w(x) \sum_{j \neq m} \int_{D_{j}} \, \int_{0}^{1} \underset{x}{\nabla} \Phi_{k}(z_{m}+t(x-z_{m});z_{j}) \cdot (x-z_{m}) \, dt \, v_{j}(y) \; dy \; dx \\
&+& \int_{D_{m}} w(x) \sum_{j \neq m} \int_{D_{j}} \, \int_{0}^{1} \underset{y}{\nabla} \Phi_{k}(x;z_{j}+t(y-z_{j})) \cdot (y - z_{j}) \, dt \, v_{j}(y) \; dy \; dx.
\end{eqnarray}
In the sequel, we will derive an estimation for  $Remainder$. For this, we need the following proposition.
\begin{proposition}\label{aprioriestimateLemma}
The total field, in the presence of particles, can be estimated by the source field via the following inequality:
\begin{equation}\label{aprioriestimation3D}
\Vert v \Vert \leq \delta^{-h} \, \Vert u \Vert.
\end{equation}
The scattering coefficient satisfies the following estimations: 
\begin{equation}
C_{m} := \int_{D_{m}} w(x) \, dx = \mathcal{O}\left( \delta^{1-h} \right) \quad \text{and} \quad \left\Vert w \right\Vert_{\mathbb{L}^{2}(D_{m})} = \mathcal{O}\left( \delta^{-\frac{1}{2}-h} \right).
\end{equation}
\end{proposition}
\begin{Proof}
For the proof of this proposition, we refer the readers to Appendix \ref{appendixproposition}. \QEDB
\end{Proof}
We split $(\ref{Remainder})$ into four terms that can be estimated separately.
\bigskip
\begin{enumerate}
\item[$\ast$] Estimation of:
\begin{eqnarray}\label{estimationT1}
\nonumber 
T_{1} & := & \frac{ a_{0} }{k^{2}  \, \tau} \, \int_{D_{m}} w(x) \, \int_{0}^{1} \nabla u^{i}(z_{m}+t(x-z_{m})) \centerdot (x-z_{m}) \, dt \, dx \\ \nonumber
\left\vert T_{1} \right\vert & \lesssim &  \tau^{-1} \, \left\vert \int_{D_{m}} w(x) \, \int_{0}^{1} \nabla u^{i}(z_{m}+t(x-z_{m})) \centerdot (x-z_{m}) \, dt \, dx \right\vert \\ \nonumber
& \lesssim & \tau^{-1} \, \left\Vert w \right\Vert_{\mathbb{L}^{2}(D_{m})} \, \left\Vert \int_{0}^{1} \nabla u^{i}(z_{m}+t(\cdot -z_{m})) \centerdot (\cdot - z_{m}) \, dt \right\Vert_{\mathbb{L}^{2}(D_{m})} \\ &=& \mathcal{O}\left( \delta^{\frac{9}{2}} \,\; \parallel w  \parallel_{\mathbb{L}^{2}(D_{m})} \right), 
\end{eqnarray}
where the last estimation can be justified using the smoothness of the incident field, $u^{i}(\cdot)$, and the fact that $\tau \sim \delta^{-2}$.
\item[]
\item[$\ast$] Estimation of:
\begin{eqnarray}\label{estimationT2}
\nonumber
T_{2} &:=& \int_{D_{m}} w(x) \int_{D_{m}} \sum_{\ell \geq 1} \frac{(ik)^{\ell + 1}}{(\ell + 1)!} \, \left\vert x-y \right\vert^{\ell} v_{m}(y) \, dy \, dx \\ \nonumber
\left\vert T_{2} \right\vert & \leq & \left\Vert  w \right\Vert_{\mathbb{L}^{2}(D_{m})} \,\, \left\Vert \int_{D_{m}} \sum_{\ell \geq 1} \frac{(ik)^{\ell + 1}}{(\ell + 1)!} \, \left\vert \cdot - y \right\vert^{\ell} v_{m}(y) \, dy \right\Vert_{\mathbb{L}^{2}(D_{m})} \\
&=& \mathcal{O}\left( \delta^{4} \, \left\Vert  w \right\Vert_{\mathbb{L}^{2}(D_{m})} \, \left\Vert  v_{m} \right\Vert_{\mathbb{L}^{2}(D_{m})}\right) \overset{(\ref{aprioriestimation3D})}{=} \mathcal{O}\left( \delta^{\frac{11}{2}-h} \, \left\Vert  w \right\Vert_{\mathbb{L}^{2}(D_{m})}  \right).
\end{eqnarray}
\item[]
\item[$\ast$] Estimation of:
\begin{eqnarray}\label{estimationT3}
\nonumber
T_{3} &:=& \int_{D_{m}} w(x) \sum_{j \neq m} \int_{D_{j}} \, \int_{0}^{1} \underset{x}{\nabla} \Phi_{k}(z_{m}+t(x-z_{m});z_{j}) \cdot (x-z_{m}) \, dt \, v_{j}(y) \; dy \; dx \\ \nonumber
\left\vert T_{3} \right\vert & \leq & \left\Vert w \right\Vert_{\mathbb{L}^{2}(D_{m})} \, \sum_{j \neq m} \left\Vert \int_{D_{j}} \, \int_{0}^{1} \underset{x}{\nabla} \Phi_{k}(z_{m}+t(\cdot -z_{m});z_{j}) \cdot (\cdot -z_{m}) \, dt \, v_{j}(y) \; dy \right\Vert_{\mathbb{L}^{2}(D_{m})} \\ \nonumber
& \lesssim & \left\Vert w \right\Vert_{\mathbb{L}^{2}(D_{m})} \, \delta^{4} \, \sum_{j \neq m} \frac{1}{\left\vert z_{m} - z_{j} \right\vert^{2}} \,\, \left\Vert v_{j}\right\Vert_{\mathbb{L}^{2}(D_{m})} \\ \nonumber
& \leq & \left\Vert w \right\Vert_{\mathbb{L}^{2}(D_{m})} \, \delta^{4} \, \left[ \sum_{j \neq m} \frac{1}{\left\vert z_{m} - z_{j} \right\vert^{4}} \right]^{\frac{1}{2}} \,\, \left[ \sum_{j \neq m}  \left\Vert v_{j}\right\Vert^{2}_{\mathbb{L}^{2}(D_{m})} \right]^{\frac{1}{2}} \\
& \leq & \left\Vert w \right\Vert_{\mathbb{L}^{2}(D_{m})} \, \delta^{4} \, d^{-2} \,\, \left\Vert v \right\Vert_{\mathbb{L}^{2}(D)} \overset{(\ref{aprioriestimation3D})}{=} \mathcal{O}\left( \left\Vert w \right\Vert_{\mathbb{L}^{2}(D_{m})} \, \delta^{\frac{11}{2}-2t-h} \right).
\end{eqnarray}
\item[]
\item[$\ast$] Estimation of:
\begin{equation*}
T_{4} := \int_{D_{m}} w(x) \sum_{j \neq m} \int_{D_{j}} \, \int_{0}^{1} \underset{y}{\nabla} \Phi_{k}(x;z_{j}+t(y-z_{j})) \cdot (y - z_{j}) \, dt \, v_{j}(y) \; dy \; dx.
\end{equation*}
Because the expression of $T_{4}$ is similar to the one of $T_{3}$ and to avoid the redundancy of computations, we skip the computation steps for its estimation and we deduce that:
\begin{equation}\label{estimationT4}
T_{4} = \mathcal{O}\left( \left\Vert w \right\Vert_{\mathbb{L}^{2}(D_{m})} \, \delta^{4} \, d^{-2} \,\, \left\Vert v \right\Vert_{\mathbb{L}^{2}(D)} \right) \overset{(\ref{aprioriestimation3D})}{=} \mathcal{O}\left( \left\Vert w \right\Vert_{\mathbb{L}^{2}(D_{m})} \, \delta^{\frac{11}{2}-2t-h} \right).
\end{equation}
\end{enumerate}
Then, by gathering $(\ref{estimationT1})-(\ref{estimationT4})$, the remainder term given by  $(\ref{Remainder})$ will be estimated as: 
\begin{equation}\label{estimationRemainder}
Remainder = \mathcal{O}\left( \left\Vert w \right\Vert_{\mathbb{L}^{2}(D_{m})} \, \delta^{\min(\frac{11}{2}-2t-h;\frac{9}{2})} \right) \overset{(\ref{estimationnormw})}{=} \mathcal{O}\left(\delta^{4-h+\min(1-2t-h;0)} \right).
\end{equation}
Finally, using the estimation $(\ref{estimationRemainder})$, the algebraic system given by $(\ref{AlgebraicSystem})$ becomes: 
\begin{eqnarray*}
\left[ 1 - i \, k \, \int_{D_{m}} w(x) \, dx \right] \; \int_{D_{m}} v_{m}(x) dx  &=&  \frac{a_{0}}{k^{2} \, \tau} u^{i}(z_{m}) \int_{D_{m}} w(x) \, dx \\ &+& \int_{D_{m}} w(x) \, dx \, \sum_{j \neq m} \Phi_{k}(z_{m},z_{j}) \int_{D_{j}} v_{j}(y) \, dy \\ &+& \mathcal{O}\left(\delta^{4-h+\min(1-2t-h;0)} \right).
\end{eqnarray*}
To write short the previous system we set, for $j=1,\cdots,M$, the following notations:
\begin{equation*}
C_{j} := \int_{D_{j}} w(x) \, dx, \,\, Q_{j} := k^{2} \, \frac{1}{a_{0}} \, \tau \, C_{j}^{-1} \, \int_{D_{j}} v_{j}(x) \; dx \quad \text{and}  \quad  U^{i}(z_{j}) := \left[ 1 - \, i \, k \, C_{j} \right]^{-1} u^{i}(z_{j}). 
\end{equation*}
Then, the previous system will be reduced to: 
\begin{equation}\label{SA3D}
Q_{m} - \left( 1 - i \, k \, C_{m} \right)^{-1} \, \sum_{j \neq m} \Phi_{k}(z_{m};z_{j}) \, C_{j} \, Q_{j} \, = U^{i}(z_{m}) + \, Error_{m}, 
\end{equation}
where 
\begin{equation}\label{Errorm}
Error_{m} = \mathcal{O}\left(  \delta^{\min(2-2t-h;1)} \right).
\end{equation}
The matrix representation of the last algebraic system, see $(\ref{SA3D})$, is as follows:
\begin{equation}\label{sysper}
\left( I - B_{k} \right) \; \overline{Q} \; = \; \overline{U} \; + \; \overline{Err},
\end{equation}
where 
\begin{equation*}
B_{k} := \begin{pmatrix} 
0 & \displaystyle\frac{C_{2} \Phi_{k}(z_{1};z_{2})}{(1-ikC_{1})}  & \cdots &   \displaystyle\frac{C_{M} \Phi_{k}(z_{1};z_{M})}{(1-ikC_{1})} \\
 \displaystyle\frac{C_{1} \Phi_{k}(z_{2};z_{1})}{(1-ikC_{2})} & 0 & \cdots &  \displaystyle\frac{C_{M} \Phi_{k}(z_{2};z_{M})}{(1-ikC_{2})}
\\
\vdots & \vdots & \ddots & \vdots \\
 \displaystyle\frac{C_{1} \Phi_{k}(z_{M};z_{1})}{(1-ikC_{M})} &  \displaystyle\frac{C_{2} \Phi_{k}(z_{M};z_{2})}{(1-ikC_{M})} & \cdots  & 0
\end{pmatrix} 
\end{equation*}
and 
\begin{equation*}
\overline{Q} := \left( Q_{j} \right)_{j=1}^{M}, 
\,\, \overline{U} := \left( U^{i}(z_{j}) \right)_{j=1}^{M}
\quad \text{and}  \quad
\overline{Err} := \left( Error_{j} \right)_{j=1}^{M}.
\end{equation*}
To investigate the invertibility of $(\ref{sysper})$, we start first by proving the invertibility of the corresponding non perturbed linear system, given by: 
\begin{equation}\label{sys}
\left( I - B_{k} \right) \cdot \tilde{Q} \; = \;  \overline{U},
\end{equation}
and then to justify that the difference, in the norm sense, between the solution of the perturbed and the unperturbed problem is at most of the same order as $\overline{Err}$. 

To invert the algebraic system using the Born series, we must check under what conditions $\Vert B_{k} \Vert < 1 $.
\begin{equation*}
\Vert B_{k} \Vert  :=  \max_{i=1,\cdots, M} \left\lbrace \sum_{j \neq i}  \vert  C_{j} \vert \; \frac{\vert \Phi_{k}(z_{i};z_{j})\vert}{\vert 1-ikC_{i} \vert} \right\rbrace 
  =  \max_{i=1,\cdots, M} \left\lbrace \vert  C_{j} \vert \, \sum_{j \neq i}  \frac{1}{\vert z_{i}-z_{j} \vert \, \vert 1-ikC_{i} \vert} \right\rbrace. 
\end{equation*}
As we know, from $(\ref{C=intW})$, that $C_{j} = \mathcal{O}\left( \delta^{1-h} \right), \; j=1,\cdots,M$, we approximate the previous formula by:     
\begin{equation}\label{normBk}
\Vert B_{k} \Vert =  \delta^{1-h} \; \sum_{j \neq i}  \; \frac{1}{\vert z_{i}-z_{j} \vert} = \mathcal{O}\left(\delta^{1-h} \; d^{-1} \right) = \mathcal{O}\left( \delta^{1-h-t} \right).
\end{equation} 
Hence, we deduce that:
\begin{equation}\label{constantlesssmall}
\Vert B_{k} \Vert < 1 \Leftrightarrow 1-h-t >0.
\end{equation}
Now for the critical case, i.e. $1-h-t = 0$, we put the following assumptions: 
\begin{equation*}
\sum_{j \neq i}  \; \frac{1}{\vert z_{i}-z_{j} \vert} = d_{0} \;\, d^{-1} = d_{0} \, \delta^{-t} \quad \text{and} \quad C_{j} =  C_{0} \;\, \delta^{1-h}. 
\end{equation*}
Under these assumptions and the condition $1-h-t = 0$, we deduce that:  
\begin{equation}\label{constantless1}
\Vert B_{k} \Vert < 1 \Leftrightarrow C_{0} \, d_{0} \, < 1.
\end{equation}

We have justified that under the conditions $(\ref{constantlesssmall})$ or $(\ref{constantless1})$, for the critical case, the invertibility of the unperturbed system, see $(\ref{sys})$. 
\bigskip
\newline
Now, under the conditions of invertibility, by subtracting $(\ref{sysper})$ from $(\ref{sys})$ we obtain: 
\begin{equation*}
\overline{Q} - \tilde{Q} = \left( I - B_{k} \right)^{-1} \; \overline{Err},
\end{equation*}
and by taking the Euclidian norm both side of the equality, we get: 
\begin{equation}\label{unper-per}
\left\Vert \overline{Q} - \tilde{Q} \right\Vert = \left\Vert \left( I - B_{k} \right)^{-1} \; \overline{Err} \right\Vert \lesssim  \left\Vert \overline{Err} \right\Vert := \left[ \sum_{m=1}^{M} \left\vert Error_{m} \right\vert^{2} \right]^{\frac{1}{2}} \overset{(\ref{Errorm})}{=} \mathcal{O}\left(  \delta^{\min(2-2t-h;1)} \right). 
\end{equation}
To prove $(\ref{intro-usca})$ we will compute an approximation for the scattered field. More precisely, we start by setting the following L.S.E:
\begin{equation*}
u(x) - k^{2} \, \frac{1}{a_{0}} \, \tau \, \int_{D} \Phi_{k}(x,y) \, u(y) \, dy = u^{i}(x),  
\end{equation*} 
which can be rewritten, knowing that $u(\cdot) = u^{s}(\cdot) + u^{i}(\cdot)$ , as:
\begin{equation*}
u^{s}(x) = k^{2} \, \frac{1}{a_{0}} \, \tau \, \int_{D} \Phi_{k}(x,y) \, u(y) \, dy = k^{2} \, \frac{1}{a_{0}} \, \tau \, \sum_{m=1}^{M} \, \int_{D_{m}} \Phi_{k}(x,y) \, v_{m}(y) \, dy,  
\end{equation*}
where $v_{m}(\cdot) = u(\cdot)_{|_{D_{m}}}$, and we have assumed that $x$ is away from $D = \underset{m=1}{\overset{M}{\cup}} D_{m}$. Then, 
using Taylor expansion for the function $\Phi_{k}(x,\cdot)$ we obtain:    
\begin{eqnarray*}
u^{s}(x,\theta) & = & k^{2} \, \frac{1}{a_{0}} \,  \tau \sum_{m=1}^{M} \Phi_{k}(x,z_{m}) \; \int_{D_{m}} v_{m}(y) \, dy \\ &+& k^{2} \, \frac{1}{a_{0}} \, \tau \, \sum_{m=1}^{M} \, \int_{D_{m}}  \int_{0}^{1}  \underset{y}{\nabla} \Phi_{k}(x,z_{m}+t(y-z_{m})) \cdot \left(y - z_{m} \right) dt  \, v_{m}(y) \, dy.
\end{eqnarray*}
The second term on the right hand side can be estimated as: 
\begin{eqnarray*}
I_{1} &:=& k^{2} \, \frac{1}{a_{0}} \, \tau \, \sum_{m=1}^{M} \, \int_{D_{m}}  \int_{0}^{1}  \underset{y}{\nabla} \Phi_{k}(x,z_{m}+t(y-z_{m})) \cdot \left(y - z_{m} \right) dt  \, v_{m}(y) \, dy \\
\left\vert I_{1} \right\vert & \lesssim & \delta^{-2} \, \sum_{m=1}^{M} \left\vert \int_{D_{m}}  \int_{0}^{1}  \underset{y}{\nabla} \Phi_{k}(x,z_{m}+t(y-z_{m})) \cdot \left(y - z_{m} \right) dt  \, v_{m}(y) \, dy \right\vert \\
& \lesssim & \delta^{-2} \, \sum_{m=1}^{M} \left\Vert   \int_{0}^{1}  \underset{y}{\nabla} \Phi_{k}(x,z_{m}+t(\cdot -z_{m})) \cdot \left(\cdot - z_{m} \right) dt \right\Vert_{\mathbb{L}^{2}(D_{m})} \, \left\Vert v_{m} \right\Vert_{\mathbb{L}^{2}(D_{m})} \\
& \lesssim & \delta^{\frac{1}{2}} \, \sum_{m=1}^{M}  \left\Vert v_{m} \right\Vert_{\mathbb{L}^{2}(D_{m})} \leq  \delta^{\frac{1}{2}} \, M^{\frac{1}{2}} \, \left( \sum_{m=1}^{M}  \left\Vert v_{m} \right\Vert^{2}_{\mathbb{L}^{2}(D_{m})} \right)^{\frac{1}{2}} = \mathcal{O}\left( \delta^{\frac{1}{2}} \,  \left\Vert v \right\Vert_{\mathbb{L}^{2}(D)} \right).
\end{eqnarray*}
Then,  
\begin{equation*}
I_{1} = \mathcal{O}\left( \delta^{\frac{1}{2}} \,  \left\Vert v \right\Vert_{\mathbb{L}^{2}(D)} \right) \overset{(\ref{aprioriestimation3D})}{=} \mathcal{O}\left( \delta^{2-h} \,  \right).
\end{equation*}
\begin{eqnarray*}
u^{s}(x,\theta) & = &  k^{2} \, \frac{1}{a_{0}} \, \tau \sum_{m=1}^{M} \Phi_{k}(x,z_{m}) \; \int_{D_{m}} v_{m}(y) \, dy  + \mathcal{O}\left( \delta^{2-h} \right) \\ &=& \sum_{m=1}^{M} \Phi_{k}(x,z_{m}) \; C_{m} \; Q_{m} \, + \mathcal{O}\left( \delta^{2-h} \right) \\  
& = &  \langle \Phi_{k}(x,z) \; C ; \overline{Q} \rangle +  \mathcal{O}\left( \delta^{2-h} \right) \\ 
& = & \langle \Phi_{k}(x,z) \; C ; \tilde{Q} \rangle + \langle \Phi_{k}(x,z) \; C ; ( \overline{Q} - \tilde{Q} ) \rangle +  \mathcal{O}\left(\delta^{2-h} \right). 
\end{eqnarray*}
For the second term on the right hand side, we have: 
\begin{eqnarray*}
I_{2} & := & \langle \Phi_{k}(x,z) \; C ; ( \overline{Q} - \tilde{Q} ) \rangle \\
\left\vert I_{2} \right\vert & \leq & \left\Vert  \Phi_{k}(x,z) \; C \right\Vert \; \left\Vert \overline{Q} - \tilde{Q} \right\Vert \overset{(\ref{unper-per})}{\lesssim} \left\Vert  C \right\Vert \; \delta^{\min(2-2t-h;1)}.
\end{eqnarray*}
Thanks to $(\ref{C=intW})$, which gives us an estimation for the scattering coefficient $C_{m} \sim \delta^{1-h}$, we obtain: 
\begin{equation*}
\left\vert I_{2} \right\vert \lesssim M^{\frac{1}{2}} \, \delta^{1-h} \; \delta^{\min(2-2t-h;1)} = \mathcal{O}\left( \delta^{1 - h +\min(2-2t-h;1)} \right).
\end{equation*} 
Then, 
\begin{eqnarray}\label{SF}
\nonumber
u^{s}(x,\theta) &=&  \langle \Phi_{k}(x,z) \; C ; \tilde{Q} \rangle  + \mathcal{O}\left( \delta^{1 - h +\min(2-2t-h;1)} \right) + \mathcal{O}\left( \delta^{2-h} \right) \\
& = & \langle \Phi_{k}(x,z) \; C ; \tilde{Q} \rangle  + \mathcal{O}\left( \delta^{\min(3-2t-2h;2-h)} \right).
\end{eqnarray}
This proves $(\ref{intro-usca})$.
We recall from $(\ref{sysper})$, that we have: 
\begin{equation*}
\left( I - B_{k} \right) \; \overline{Q} \; = \; \overline{U} \; + \; \overline{Err}.
\end{equation*} 
Then: 
\begin{equation*}
\overline{Q}=\left( I - B_{k} \right)^{-1}  \overline{U}  + \left( I - B_{k} \right)^{-1}  \overline{Err} =  \sum_{n \geq 0} B_{k}^{n} \; \overline{U}  +  \overline{Err}. 
\end{equation*}
We have:
\begin{equation*}
\Vert B_{k}^{n} \; \overline{U} \Vert \leq \Vert B_{k}\Vert^{n} \; \Vert \overline{U} \Vert \simeq \Vert B_{k}\Vert^{n} \overset{(\ref{normBk})}{=} \mathcal{O}\left( \delta^{n(1-h-t)} \right).
\end{equation*} 
Now, we define the truncated solution of the unperturbed problem $\tilde{Q}^{N}$, $N \in \mathbb{N}$, by:
\begin{equation*}
\tilde{Q}^{N} := \sum_{n=0}^{N} \, B_{k}^{n} \cdot U
\end{equation*}
and its associated scattered field
\begin{equation}\label{TSF}
u^{s,N}(x,\theta)  :=  \langle \Phi_{k}(x,z) \; C ; \tilde{Q}^{N} \rangle. 
\end{equation}
Subtracting $(\ref{TSF})$ from $(\ref{SF})$, we obtain: 
\begin{eqnarray}\label{219}
\nonumber
u^{s}(x,\theta) - u^{s,N}(x,\theta) & = &  \langle \Phi_{k}(x,z) \; C ; \tilde{Q} - \tilde{Q}^{N} \rangle  + \mathcal{O}\left( \delta^{\min(3-2t-2h;2-h)} \right) \\ \nonumber
\left\vert u^{s}(x,\theta) - u^{s,N}(x,\theta) \right\vert & = & \left\vert  \langle \Phi_{k}(x,z) \; C ; \sum_{n \geq N+1} \, B_{k}^{n} \cdot U  \rangle  \right\vert + \mathcal{O}\left( \delta^{\min(3-2t-2h;2-h)} \right) \\ \nonumber
& \lesssim & \frac{\vert C \vert}{\underset{j=1,\cdots , M}{\min} \vert x-z_j \vert} \ \;  \sum_{n \geq N+1} \, \delta^{n (1-t-h)}  + \mathcal{O}\left( \delta^{\min(3-2t-2h;2-h)} \right) \\
& = & \delta^{1-h} \; \delta^{(N+1)(1-t-h)}  + \mathcal{O}\left( \delta^{\min(3-2t-2h;2-h)} \right).
\end{eqnarray}
We choose $h$ such that
\begin{equation*}
(1-h)+(N+1)(1-h-t) \leq  \min\{2-h,\; 3-2h-2t\},
\end{equation*} 
which can be rewritten as
\begin{equation}\label{condition-Foldy-Lax-proof}
0\leq 1-h-t \leq \min\left\{ \frac{1}{N+1}; \frac{1-t}{N}\right\},
\end{equation}
hence, $(\ref{219})$ becomes, 
\begin{equation}\label{N-interactions-proof}
u^{s}(x,\theta) - u^{s,N}(x,\theta) =  \mathcal{O}\left( \delta^{(1-h) + (N+1)(1-t-h)} \right).
\end{equation}
This justifies $(\ref{intro-N-interactions-proof})$. In addition, 
\begin{eqnarray}\label{N-interactions-scattered-field}
\nonumber
u^{s,N}(x,\theta)-u^{s,N-1}(x,\theta) & \overset{(\ref{TSF})}{=} & \langle \Phi_{k}(x,z) \, C; B_{k}^{N} \cdot U \rangle  \\ 
& \sim & \delta^{(1-h) + N(1-t-h)} >>  \delta^{(1-h) + (N+1)(1-t-h)} 
\end{eqnarray}
for in any bounded domain away from the collection of centers $z_j, j=1, \cdots, M$. Here $u^{s,N}(x,\theta)$ is the scattered field 
after $N$-interactions between the particles.  
\bigskip

Under the limit condition $1-t-h=0$, the expansion in $(\ref{SF})$ is reduced to: 
\begin{equation*}
u^{s}(x,\theta)  =  \langle \Phi_{k}(x,z) \; C ; \tilde{Q} \rangle  + \mathcal{O}\left( \delta \right).
\end{equation*}
If we set $u^{s,\infty}(x,\theta) :=  \langle \Phi_{k}(x,z) \; C ; \tilde{Q} \rangle$, then we obtain:
\begin{equation}\label{SFCC}
u^{s}(x,\theta)  - u^{s,\infty}(x,\theta) = \mathcal{O}\left( \delta \right)
\end{equation}
which proves $(\ref{intro-infty-interactions-proof})$. Actually, from the definition of $u^{s}(x,\theta)$, we have  
\begin{equation*}
u^{s,\infty}(x,\theta):=  \langle \Phi_{k}(x,z) \; C ; \sum_{n=0}^{\infty} \, B_{k}^{n} \cdot U \rangle = \sum_{n=0}^{\infty} \, \langle \Phi_{k}(x,z) \; C ; B_{k}^{n} \cdot U \rangle \sim \delta^{1-h} \sim \delta^{t}.
\end{equation*}
 is the field generated after all the interaction between the particles. We call it the Foldy-Lax field. In addition, it is clear that for each term in the previous series, i.e. any order $n \in \mathbb{N}$, we have $\langle \Phi_{k}(x,z) \; C ; B_{k}^{n} \cdot U \rangle \sim \delta^{1-h} >> \delta $, and this allows us to justify that, at any order of interactions $n \in \mathbb{N}$, at least one interaction between the $M$ particles can be seen from the scattered field measured away from the cluster of particles. This ends the proof of Corollary \ref{coro12}.
\bigskip
\newline
The justification of Corollary \ref{coro13}, can be handled in similar manner by taking into account the fact that $t=0$. 


\section{Proof of Theorem \ref{2D-case}}
We recall that the fundamental solution, $\Phi_{k}\left( \cdot,\cdot \right)$, of the  Helmholtz equation in dimension two  satisfies 
\begin{center}
$\begin{cases}
\underset{y}{\Delta} \Phi_{k}\left( x , y \right) + k^{2} \, \frac{b_{0}}{a_{0}} \, 
\Phi_{k}\left( x , y \right) = - \underset{x}{\delta}(y), \quad x,y \in \mathbb{R}^{2}, & \\
\Phi_{k}\left( \cdot,\cdot \right) \;\; S.R.C \;\; \text{at} \;\; infinity,   &
\end{cases}$
\end{center}
where $k$ is such that $k^{2} = k^{2} \, \frac{b_{0}}{a_{0}}$ and it is given via the Hankel function of first kind $H_{0}^{(1)}$. The following lemma, on the decomposition of the Green's kernel, is of importance to reduce the complexities of the Lippmann-Schwinger equation, and then to invert it.
\begin{lemma}\label{lemmadecomposition}
We have the following asymptotic expansion 
\begin{equation}\label{2Dlogexpansion}
\Phi_{k}(x , y) =   \Phi_{0}(x, y)  + \bm{E} + \mathcal{O}(\vert x - y \vert^{2} \; \log(\vert x-y \vert)), \,\, x \,\, near \,\, y,
\end{equation}
where $\Phi_{0}(\cdot,\cdot)$ is the fundamental solution of the Laplace equation in dimension two given by:
\begin{equation*}
\Phi_{0}(x, y) := \frac{-1}{2 \, \pi} \, \log\left(\left\vert x - y \right\vert \right), \quad x \neq y,
\end{equation*}
and $\bm{E}$ is the constant given by: 
\begin{equation*}
\bm{E} := \frac{i}{4} - \frac{1}{2 \pi} \left[\log\left( \frac{k}{2} \right) + \lim_{p \rightarrow +\infty} \left( \sum_{m=1}^{p} \frac{1}{m} - \log(p) \right) \right].
\end{equation*}
\end{lemma}
\begin{proof} 
See formula $(3.84)$, page 74 in \cite{colton1998inverse}.
\end{proof}
We start our proof of Theorem \ref{2D-case}, by setting the following Lippmann-Schwinger equation: 
\begin{equation*}
v_{m}(x) - k^{2} \, \frac{1}{a_{0}} \, \tau \, \int_{D_{m}} \Phi_{k}(x , y) \, v_{m}(y) \, dy - k^{2} \, \frac{1}{a_{0}} \, \tau \, \sum_{j \neq m} \int_{D_{j}} \Phi_{k}(x , y) \, v_{j}(y) \, dy = u^{i}(x), \quad x \in D_{m},
\end{equation*}
where, we assume that\footnote{As in 3D case, to simplify the computations, we write the detailed proof with $\tau$, but of course we can do it with $\left\{ \tau_{j} \right\}_{j=1}^{M}$.}, $\tau_{j} = \tau, \;\; j=1,\cdots,M$ as well. Then, thanks to Lemma \ref{lemmadecomposition}, we rewrite the previous L.S.E as: 
\begin{eqnarray}\label{equa1}
\nonumber
v_{m}(x) &-& k^{2} \, \frac{1}{a_{0}} \, \tau \, \int_{D_{m}} \Phi_{0}(x, y) \, v_{m}(y) \, dy  - k^{2} \, \frac{1}{a_{0}} \, \tau \, \bm{E} \, \int_{D_{m}} v_{m}(y) \, dy \\ \nonumber &-& k^{2} \, \frac{1}{a_{0}} \, \tau \, \sum_{j \neq m} \int_{D_{j}} \Phi_{k}(x , y) \, v_{j}(y) \, dy \\ &=& u^{i}(x) + \tau \, \int_{D_{m}}  \mathcal{O}(\vert  x - y \vert^{2} \; \log(\vert x - y \vert)) \, v_{m}(y) \, dy.
\end{eqnarray}
We introduce the Newtonian Potential operator $A^{0}$, defined by: 
\begin{equation}\label{Newpot2D}
A^{0}\left(v_{m}\right)(x) = \int_{D_{m}} \frac{-1}{2\pi} \log(\vert x-y \vert) \; v_{m}(y) \, dy. 
\end{equation}
Then, going back to $(\ref{equa1})$, using the definition of the Newtonian Potential operator, see $(\ref{Newpot2D})$, and expanding the Green's kernel and the incident field near the points $\left( z_{i} \right)_{i=1}^{M}$, we get:
\begin{eqnarray}\label{equa1}
\nonumber
\left[I - k^{2} \, \frac{1}{a_{0}} \, b \, A^{0} \right]\left( v_{m} \right)(x) &-& k^{2} \, \frac{1}{a_{0}} \; b \; \bm{E}  \int_{D_{m}}  v_{m}(y) \, dy \\ &-& k^{2} \, \frac{1}{a_{0}} \, b \, \sum_{j \neq m} \Phi_{k}(z_{m} , z_{j}) \,  \int_{D_{j}} v_{j}(y) \, dy = u^{i}(z_{m}) + Err_{1}(x), 
\end{eqnarray}
where 
\begin{eqnarray}\label{Err1}
\nonumber
Err_{1}(x) &:=& \int_{0}^{1} \, \nabla u^{i}(z_{m} + t (x-z_{m})) \cdot (x-z_{m}) \, dt \\ \nonumber &+& \tau \, \int_{D_{m}}  \mathcal{O}(\vert x-y\vert^{2} \; \log(\vert x-y \vert)) \, v_{m}(y) \, dy  \\ \nonumber
&+& k^{2} \, \frac{1}{a_{0}} \, \tau \, \sum_{j \neq m} \int_{D_{j}} \int_{0}^{1} \underset{y}{\nabla} \Phi_{k}(x,z_{j}+t(y-z_{j})) \cdot (y-z_{j}) \, dt \, v_{j}(y) \, dy \\ \nonumber
&+& k^{2} \, \frac{1}{a_{0}} \, \tau \, \sum_{j \neq m} \int_{0}^{1} \underset{x}{\nabla}\Phi_{k}(z_{m} + t (x-z_{m}), z_{j}) \cdot (x-z_{m}) dt \, \int_{D_{j}} v_{j}(y) \, dy \\ \nonumber
&-& k^{2} \, \frac{1}{a_{0}} \, b_{0}(z_{m}) \, A^{0} \left( v_{m} \right)(x) - k^{2} \, \frac{1}{a_{0}} \, b_{0}(z_{m}) \; \bm{E} \;  \int_{D_{m}}  v_{m}(y) \, dy \\
&-& k^{2} \, \frac{1}{a_{0}} \, b_{0}(z_{m}) \, \sum_{j \neq m} \Phi_{k}(z_{m} , z_{j})  \int_{D_{j}} v_{j}(y) \, dy
\end{eqnarray}
Let $w$ to be:  
\begin{equation*}
w := k^{2} \, \frac{1}{a_{0}} \, \epsilon_{p} \, \left[ I - k^{2} \, \frac{1}{a_{0}} \, \epsilon_{p} \, A^{0} \right]^{-1}(1) = \left[ \frac{a_{0}}{k^{2}  \, b} I - A^{0} \right]^{-1}(1) 
\end{equation*}
and its corresponding scattering coefficient $C_{m}$ defined as: 
\begin{equation*}
C_{m} := \int_{D_{m}} w(x) \, dx. 
\end{equation*}
Then successively, for the equation $(\ref{equa1})$, dividing both sides by $k^{2} \,\frac{b}{a_{0}}$, taking  the inverse operator of $\left[ \displaystyle\frac{a_{0}}{k^{2} \, b} I - A^{0} \right]$ and integrating over $D_{m}$, we obtain
\begin{eqnarray*} 
\left[ 1 - C_{m} \; \bm{E} \right] \int_{D_{m}} v_{m}(y) \, dy &-&  C_{m} \; \sum_{j \neq m} \Phi_{k}(z_{m}, z_{j}) \,  \int_{D_{j}} v_{j}(y) \, dy  \\ & = & \frac{C_{m} \, a_{0}}{k^{2} \, b} u^{i}(z_{m}) + \frac{a_{0}}{k^{2} \, b}  \int_{D_{m}} w(x) Err_{1}(x) \, dx.
\end{eqnarray*}
By multiplying the last equation by $k^{2} \; \frac{b}{a_{0}} \; C_{m}^{-1}$, we end up with:
\begin{equation}\label{AS2D}
k^{2} \,  \frac{b}{a_{0}} \,  \left[ C^{-1}_{m} -  \; \bm{E}  \right] \int_{D_{m}} v_{m}(y) \, dy -  k^{2} \,  \frac{b}{a_{0}} \; \sum_{j \neq m} \Phi_{k}(z_{m}, z_{j}) \, \int_{D_{j}} v_{j}(y) \, dy  =  u^{i}(z_{m}) + Err_{2,m},
\end{equation}
where:
\begin{equation}\label{Err2}
Err_{2,m} :=   C_{m}^{-1} \,   \int_{D_{m}} w(x) Err_{1}(x) \, dx.
\end{equation}
To estimate the term $Err_{2,m}$, we use the coming Lemma.
\begin{lemma}\label{apestimate2D}
The total field $v(\cdot)$ admits the following a priori estimation:
\begin{equation}\label{prioriest}
\Vert v \Vert_{\mathbb{L}^{2}(D)} \leq \vert \log(\delta) \vert^{h} \, \Vert u \Vert_{\mathbb{L}^{2}(D)}.
\end{equation} 
For the scattering function and scattering coefficient, we have the following relations:
\begin{equation}\label{W2DC}
\Vert w \Vert_{\mathbb{L}^{2}(D)} \leq \delta^{-1} \, \vert \log(\delta) \vert^{h-1} \quad \text{and} \quad
C_{m} = \vert \log(\delta) \vert^{h-1}.
\end{equation} 
\end{lemma}
\begin{proof}
The a priori estimate $(\ref{prioriest})$ is already proved in Section 4 of \cite{2DGHANDRICHE}. If in $(\ref{prioriest})$, we take $u(\cdot) = \tau$, recalling that $\tau \sim \delta^{-2} \, \left\vert \log(\delta) \right\vert^{-1}$, we deduce the first part of $(\ref{W2DC})$. The second part can be proved using Cauchy-Schwartz inequality and the obtained estimation for $\Vert w \Vert_{\mathbb{L}^{2}(D)}$.  
\end{proof}
First, by combining $(\ref{Err1})$ and $(\ref{Err2})$, we derive the following expression for $Err_{2,m}$, 
\begin{eqnarray*}
Err_{2,m} &=& C_{m}^{-1} \, \int_{D_{m}} w(x) \, \int_{0}^{1} \, \nabla u^{i}(z_{m} + t (x-z_{m})) \cdot (x-z_{m}) \, dt \, dx \\ \nonumber &+& C_{m}^{-1} \, \tau \, \int_{D_{m}} w(x)  \int_{D_{m}} \vert x-y\vert^{2} \; \log(\vert x-y \vert) \, v_{m}(y) \, dy \, dx \\ \nonumber
&+& \frac{k^{2}  \, \tau}{C_{m} \, a_{0}} \, \int_{D_{m}} w(x)  \sum_{j \neq m} \int_{D_{j}} \int_{0}^{1} \underset{y}{\nabla} \Phi_{k}(x , z_{j}+t(y-z_{j}) ) \cdot (y-z_{j}) \, dt \, v_{j}(y) \, dy \, dx \\ 
&+& \frac{k^{2} \, \tau}{C_{m} \, a_{0}} \, \int_{D_{m}} w(x)  \sum_{j \neq m} \int_{0}^{1} \underset{x}{\nabla}\Phi_{k}( z_{m} + t (x-z_{m}) , z_{j}) \cdot (x-z_{m}) dt \, \int_{D_{j}} v_{j}(y) \, dy \, dx \\
&-& k^{2} \, \frac{1}{a_{0}} \, b_{0}(z_{m}) \, C^{-1}_{m} \, \int_{D_{m}} w(x) \cdot A^{0}(v_{m})(x) \, dx - k^{2} \, \frac{1}{a_{0}} \, b_{0}(z_{m}) \;\bm{E}  \; \int_{D_{m}} v_{m}(y) \, dy \\
&-& k^{2} \, \frac{1}{a_{0}} \, b_{0}(z_{m}) \, \sum_{j \neq m} \Phi_{k}\left(z_{m} , z_{j}\right) \, \int_{D_{j}} v_{j}(y) \, dy.
\end{eqnarray*}
Next, we split the expression of $Err_{2,m}$ into seven terms, and then we estimate them separately to get an estimation of $Err_{2,m}$.  
\begin{enumerate}
\item[]
\item[$\ast)$]Estimation of $T_{1}$:
\begin{eqnarray}\label{T12D}
\nonumber
T_{1} &:=& C_{m}^{-1} \, \int_{D_{m}} w(x) \, \int_{0}^{1} \, \nabla u^{i}(z_{m} + t (x-z_{m})) \cdot (x-z_{m}) \, dt \, dx \\ \nonumber
\left\vert T_{1} \right\vert & \overset{(\ref{W2DC})}{\lesssim} & \left\vert \log(\delta) \right\vert^{1-h} \, \left\Vert w \right\Vert_{\mathbb{L}^{2}(D_{m})} \, \left\Vert \int_{0}^{1} \, \nabla u^{i}(z_{m} + t (\cdot - z_{m})) \cdot (\cdot - z_{m}) \, dt \right\Vert_{\mathbb{L}^{2}(D_{m})} \\
& \lesssim & \left\vert \log(\delta) \right\vert^{1-h}  \, \left\Vert w \right\Vert_{\mathbb{L}^{2}(D_{m})} \, \delta^{2} \overset{(\ref{W2DC})}{=} \mathcal{O}\left( \delta \right).
\end{eqnarray} 
\item[]
\item[$\ast)$]Estimation of $T_{2}$:
\begin{eqnarray}\label{T22D}
\nonumber
T_{2} &:=& C_{m}^{-1} \, \tau \, \int_{D_{m}} w(x)  \int_{D_{m}} \vert x-y\vert^{2} \; \log(\vert x-y \vert) \, v_{m}(y) \, dy \, dx \\ \nonumber
\left\vert T_{2} \right\vert & \overset{(\ref{W2DC})}{\lesssim} & \delta^{-3} \, \left\vert \log(\delta) \right\vert^{-1} \, \left\Vert  \int_{D_{m}} \vert \cdot - y \vert^{2} \; \log(\vert \cdot - y \vert) \, v_{m}(y) \, dy \right\Vert_{\mathbb{L}^{2}(D_{m})} \\
& \lesssim & \delta \, \left\Vert v_{m} \right\Vert_{\mathbb{L}^{2}(D_{m})} \overset{(\ref{prioriest})}{=} \mathcal{O}\left( \delta^{2} \, \left\vert \log(\delta) \right\vert^{h} \right). 
\end{eqnarray}
\item[]
\item[$\ast)$] Estimation of $T_{3}$:
\begin{eqnarray*}
T_{3} &:=& \frac{k^{2} \, \tau}{C_{m} \, a_{0}} \, \int_{D_{m}} w(x)  \sum_{j \neq m} \int_{D_{j}} \int_{0}^{1} \underset{y}{\nabla} \Phi_{k}(x, z_{j}+t(y-z_{j}) ) \cdot (y-z_{j}) \, dt \, v_{j}(y) \, dy \, dx \\
\left\vert T_{3} \right\vert & \lesssim &  \frac{\left\Vert w \right\Vert_{\mathbb{L}^{2}(D_{m})}}{\delta \, \left\vert \log(\delta) \right\vert^{h}} \,  \sum_{j \neq m} \left[ \int_{D_{m}} \int_{D_{j}}  \int_{0}^{1} \left\vert \underset{y}{\nabla} \Phi_{k}(x, z_{j}+t(y-z_{j}))  \right\vert^{2} dt  dy \, dx \right]^{\frac{1}{2}} \, \left\Vert v_{j} \right\Vert_{\mathbb{L}^{2}(D_{j})}. 
\end{eqnarray*}
Thanks to Lemma \ref{lemmadecomposition}, we can approximate the previous estimation by:  
\begin{eqnarray*}
\left\vert T_{3} \right\vert & \lesssim &  \frac{\left\Vert w \right\Vert_{\mathbb{L}^{2}(D_{m})}}{\delta \, \left\vert \log(\delta) \right\vert^{h}} \,  \sum_{j \neq m} \left[ \int_{D_{m}} \int_{D_{j}}  \int_{0}^{1} \frac{1}{\left\vert  x - (z_{j}+t(y-z_{j}))\right\vert^{2}} dt  dy \, dx \right]^{\frac{1}{2}} \, \left\Vert v_{j} \right\Vert_{\mathbb{L}^{2}(D_{j})} \\
& \lesssim &  \frac{\delta \, \left\Vert w \right\Vert_{\mathbb{L}^{2}(D_{m})}}{\left\vert \log(\delta) \right\vert^{h}} \,  \sum_{j \neq m}  \frac{1}{\left\vert  z_{m} - z_{j}\right\vert}  \, \left\Vert v_{j} \right\Vert_{\mathbb{L}^{2}(D_{j})} \\ 
& \lesssim &  \frac{\delta \, \left\Vert w \right\Vert_{\mathbb{L}^{2}(D_{m})}}{\left\vert \log(\delta) \right\vert^{h}} \, \left( \sum_{j \neq m}  \frac{1}{\left\vert  z_{m} - z_{j}\right\vert^{2}} \right)^{\frac{1}{2}} \,  \left\Vert v \right\Vert_{\mathbb{L}^{2}(D)} \\ 
& \lesssim &  \frac{\delta \, \left\Vert w \right\Vert_{\mathbb{L}^{2}(D_{m})}}{\left\vert \log(\delta) \right\vert^{h}} \, d^{-1} \,  \left\Vert v \right\Vert_{\mathbb{L}^{2}(D)}.
\end{eqnarray*}
Using the a priori estimate $(\ref{prioriest})$ and the estimation of $\left\Vert w \right\Vert_{\mathbb{L}^{2}(D_{m})}$, see $(\ref{W2DC})$, we get: 
\begin{equation}\label{T42D}
\left\vert T_{3} \right\vert = \mathcal{O}\left(\left\vert \log(\delta) \right\vert^{h-1} \; \delta \; d^{-1}  \right).
\end{equation}
\item[]
\item[$\ast)$] Estimation of $T_{4}$: 
\begin{equation*}
T_{4} := \frac{k^{2} \, \tau}{C_{m} \, a_{0}} \, \int_{D_{m}} w(x)  \sum_{j \neq m} \int_{0}^{1} \underset{x}{\nabla}\Phi_{k}(z_{m} + t (x-z_{m}), z_{j} ) \cdot (x-z_{m}) dt \, \int_{D_{j}} v_{j}(y) \, dy \, dx.
\end{equation*}
The above expression for $T_{4}$ is almost the same as that of $T_{3}$, consequently its estimation can be handled in similar manner and we deduce, from $(\ref{T42D})$, that: 
\begin{equation*}
\left\vert T_{4} \right\vert = \mathcal{O}\left(\left\vert \log(\delta) \right\vert^{h-1} \; \delta \; d^{-1}  \right).
\end{equation*}
\item[$\ast)$] Estimation of $T_{5}$:
\begin{eqnarray*}
T_{5} &:=& - k^{2} \, \frac{1}{a_{0}} \, b_{0}(z_{m}) \, C^{-1}_{m} \, \int_{D_{m}} w(x) \cdot A^{0}(v_{m})(x) \, dx \\
\left\vert T_{5} \right\vert & \lesssim & \left\vert \log(\delta) \right\vert^{1-h} \, \left\Vert w \right\Vert_{\mathbb{L}^{2}(D_{m})} \,  \left\Vert A^{0}(v_{m}) \right\Vert_{\mathbb{L}^{2}(D_{m})} \\
& \lesssim & \left\vert \log(\delta) \right\vert^{2-h} \, \delta^{2} \, \left\Vert w \right\Vert_{\mathbb{L}^{2}(D_{m})} \,  \left\Vert v_{m} \right\Vert_{\mathbb{L}^{2}(D_{m})}.
\end{eqnarray*}
Then, thanks to Lemma \ref{apestimate2D}, we deduce that: 
\begin{equation*}
T_{5} = \mathcal{O}\left( \left\vert \log(\delta) \right\vert^{h+1} \, \delta^{2} \right).
\end{equation*}
\item[$\ast)$] Estimation of $T_{6}$:
\begin{eqnarray*}
T_{6} &:=& - k^{2} \, \frac{1}{a_{0}} \, b_{0}(z_{m}) \; \bm{E} \; \int_{D_{m}} v_{m}(y) \, dy \\
\left\vert T_{6} \right\vert & \lesssim & \left\Vert 1 \right\Vert_{\mathbb{L}^{2}(D_{m})} \, \left\Vert v_{m} \right\Vert_{\mathbb{L}^{2}(D_{m})} \overset{(\ref{prioriest})}{=} \mathcal{O}\left(\delta^{2} \, \left\vert \log(\delta) \right\vert^{h} \right).
\end{eqnarray*}
\item[$\ast)$] Estimation of $T_{7}$:
\begin{eqnarray*}
T_{7} &:=& - k^{2} \, \frac{1}{a_{0}} \, b_{0}(z_{m}) \, \sum_{j \neq m} \Phi_{k}\left(z_{m},  z_{j} \right) \, \int_{D_{j}} v_{j}(y) \, dy \\
\left\vert T_{7} \right\vert & \lesssim & \sum_{j \neq m} \log\left(\frac{1}{\left\vert z_{m} - z_{j} \right\vert} \right) \, \left\Vert 1 \right\Vert_{\mathbb{L}^{2}(D_{j})} \, \left\Vert v_{j} \right\Vert_{\mathbb{L}^{2}(D_{j})} \\
& \overset{(\ref{prioriest})}{\lesssim} & \delta^{2} \, \left\vert \log(\delta) \right\vert^{h} \, \sum_{j \neq m} \log\left(\frac{1}{\left\vert z_{m} - z_{j} \right\vert} \right).
\end{eqnarray*}
It is proven, in Lemma 2.6 of \cite{2DGHANDRICHE}, that: 
\begin{equation}\label{sumlog}
\sum_{j \neq m} \log\left( \frac{1}{d_{mj}} \right) = \log(1/d).  
\end{equation}
Then, 
\begin{equation}\label{T52D}
T_{7} = \mathcal{O}\left(\delta^{2} \, \left\vert \log(\delta) \right\vert^{h} \, \log(1/d) \right).
\end{equation} 
\end{enumerate}
Finally, by gathering $(\ref{T12D})-(\ref{T52D})$, we end up with the following estimation for $Err_{2,m}$. 
\begin{equation*}
Err_{2,m} = \sum_{i=1}^{7} T_{i} = \mathcal{O}\left( \delta \, \left\vert \log(\delta)\right\vert^{h-1} \, d^{-1} \right),
\end{equation*} 
and, knowing that $d \sim e^{- \, \left\vert \log(\delta) \right\vert^{t}}$, we obtain: 
\begin{equation}
Err_{2,m} =  \mathcal{O}\left( \delta \, \left\vert \log(\delta)\right\vert^{h-1} \, e^{\left\vert \log(\delta) \right\vert^{t}} \right).
\end{equation} 
At this stage, the algebraic system given by $(\ref{AS2D})$ can be written as: 
\begin{equation*}
k^{2} \, \frac{b}{a_{0}} \,  \left[ C^{-1}_{m} - \bm{E}  \right] \int_{D_{m}} v_{m}(y) \, dy -  k^{2} \, \frac{b}{a_{0}} \; \sum_{j \neq m} \Phi_{k}(z_{m}, z_{j}) \int_{D_{j}} v_{j}(y) \, dy  =  u^{i}(z_{m}) + Err_{2,m}.
\end{equation*}
We set:
\begin{equation*}
Q_{m} := k^{2} \, \frac{b}{a_{0}} \,  \left[ C^{-1}_{m} - \bm{E}  \right] \int_{D_{m}} v_{m}(y) \, dy. 
\end{equation*}
Then
\begin{equation}\label{slcc}
Q_{m} - \sum_{j \neq m} \Phi_{k}(z_{m}, z_{j}) \,\left[ C^{-1}_{j} - \bm{E} \right]^{-1}  Q_{j} = u^{i}(z_{m}) + Err_{2,m}. 
\end{equation}
To write short, we set
\begin{equation}\label{defCstar}
C_{m}^{\star} := \left[ C^{-1}_{m} - 
\bm{E}  \right]^{-1}, \;\; m=1,\cdots,M, 
\end{equation}
hence, the equation $(\ref{slcc})$ will takes the following form, 
\begin{equation}\label{slcstar}
Q_{m} - \sum_{j \neq m} \Phi_{k}(z_{m}, z_{j}) \; C_{j}^{\star} \; Q_{j} = u^{i}(z_{m}) + Err_{2,m},   
\end{equation}
which can be rewritten in a in matrix form as 
\begin{equation}\label{per}
\left( I - B_{k} \right) \cdot Q = U + Err,
\end{equation}
where $B_{k} = \left( B_{k,mj} \right)_{m,j=1}^{M}$ such that
\begin{equation}\label{defB}
B_{k,mj} := \Phi_{k}(z_{m};z_{j}) \; C_{j}^{\star} \; \left(1 - \underset{j,m}{\bm{\delta}} \right),
\end{equation}
$U = \left(u^{i}(z_{j}) \right)_{j=1,\cdots,M}$ and $Err = \left(Err_{2,m}, \cdots, Err_{2,m} \right)$. 
\bigskip
\newline
As it was done in the proof of Theorem \ref{3D-case}, we start by looking under what conditions the following non perturbed linear system
\begin{equation}\label{nonper}
\left( I - B_{k} \right) \cdot \tilde{Q} = U
\end{equation}
is invertible. To accomplish this, let us evaluate the norm of $B_{k}$. For this we have: 
\begin{equation*}
\left\Vert B_{k} \right\Vert =  \max_{m} \, \sum_{j \neq m} \left\vert B_{k,mj} \right\vert  \overset{(\ref{defB})}{=}  \max_{m} \, \sum_{j \neq m} \left\vert \Phi_{k}(z_{m}, z_{j} ) \; C_{j}^{\star} \right\vert \overset{(\ref{defCstar})}{=}  \max_{m} \, \sum_{j \neq m}  \frac{\left\vert \Phi_{k}(z_{m}, z_{j}) \right\vert}{\left\vert C^{-1}_{j} -   \bm{E}  \right\vert}. 
\end{equation*}
Thanks to $(\ref{W2DC})$, we know that $C_{j} \sim \left\vert \log(\delta) \right\vert^{h-1}$ and from  Lemma \ref{lemmadecomposition} we can approximate $\Phi_{k}(\cdot,\cdot)$ by the logarithmic kernel. We manage all this to deduce: 
\begin{equation}\label{normBlog} 
\left\Vert B_{k} \right\Vert  \lesssim   \vert \log(\delta) \vert^{h-1} \, \sum_{j \neq m} \log\left(\frac{1}{d_{mj}}\right) \overset{(\ref{sumlog})}{=} \mathcal{O}\left( \vert \log(\delta) \vert^{h-1} \,  \log\left(\frac{1}{d}\right) \right),   
\end{equation}
and, the condition $\Vert B_{k} \Vert < 1$ will be fulfilled if: 
\begin{equation*}
\log(1/d) < \vert \log(\delta) \vert^{1-h}.
\end{equation*} 
This implies that: 
\begin{equation}\label{Condinv}
d > \exp\left( -\vert \log(\delta) \vert^{1-h} \right).
\end{equation}
Now, we assume that the invertibility condition $(\ref{Condinv})$ is fulfilled, by subtracting $(\ref{nonper})$ from  $(\ref{per})$, we get
\begin{equation*}
Q - \tilde{Q} = \left( I - B_{k} \right)^{-1} \, Err \end{equation*} 
and, by taking the Euclidean norm in both sides, we obtain 
\begin{equation}\label{DiffQ-Q}
\Vert Q - \tilde{Q} \Vert\lesssim  \Vert Err \Vert = \mathcal{O}\left( Err_{2,m} \right). 
\end{equation}
\bigskip
\newline
To prove $(\ref{intro-usca-2D})$ we derive an approximation for the scattered field. To do this, we start by setting the coming L.S.E: 
\begin{equation*}
u^{s}(x,\theta)  =  k^{2} \, \frac{1}{a_{0}} \, \tau \, \sum_{m=1}^{M} \, \int_{D_{m}} \Phi_{k}(x,y) \; v_{m}(y) \; dy, 
\end{equation*}
which, after using Taylor expansion near the center $z_{m}$ for the kernel $\Phi_{k}(\cdot,\cdot)$, becomes: 
\begin{eqnarray*}
u^{s}(x,\theta) & = & k^{2} \, \frac{1}{a_{0}} \, \tau \, \sum_{m=1}^{M} \; \Phi_{k}(x,z_{m}) \; \int_{D_{m}} v_{m}(y) \, dy \\ &+& k^{2} \, \frac{1}{a_{0}} \, \tau \, \sum_{m=1}^{M} \, \int_{D_{m}} \int_{0}^{1}  \underset{y}{\nabla}\Phi_{k}(x, z_{m} + t (y - z_{m})) \cdot (y-z_{m}) \, dt \; v_{m}(y) \, dy.
\end{eqnarray*}
We estimate the second term on the right hand side as: 
\begin{eqnarray*}
J_{1} & := & k^{2} \, \frac{1}{a_{0}} \, \tau \, \sum_{m=1}^{M} \, \int_{D_{m}} \int_{0}^{1}  \underset{y}{\nabla}\Phi_{k}(x, z_{m} + t (y - z_{m})) \cdot (y-z_{m}) \, dt \; v_{m}(y) \, dy \\
\left\vert J_{1} \right\vert & \lesssim & \tau \sum_{m=1}^{M} \left\Vert \int_{0}^{1}  \underset{y}{\nabla}\Phi_{k}( x, z_{m} + t (\cdot - z_{m})) \cdot (\cdot - z_{m}) dt \right\Vert_{\mathbb{L}^{2}(D_{m})}  \left\Vert v_{m} \right\Vert_{\mathbb{L}^{2}(D_{m})}.
\end{eqnarray*}
We recall that $\tau \sim \delta^{-2} \left\vert \log(\delta) \right\vert^{-1}$ and we assume that the point $x$ is away from $D$, this implies the smoothness of the function under the integral sign. Then, 
\begin{equation*}
\left\vert J_{1} \right\vert  \lesssim   \left\vert \log(\delta) \right\vert^{-1} \, \sum_{m=1}^{M} \left\Vert v_{m} \right\Vert_{\mathbb{L}^{2}(D_{m})} \leq \left\vert \log(\delta) \right\vert^{-1}  \,  \left\Vert v \right\Vert_{\mathbb{L}^{2}(D)} \overset{(\ref{prioriest})}{=} \mathcal{O}\left( \delta \, \left\vert \log(\delta) \right\vert^{h-1}  \right).
\end{equation*}
Therefore, 
\begin{eqnarray}\label{SK} 
\nonumber
u^{s}(x,\theta) & = & k^{2} \, \frac{1}{a_{0}} \, \tau \, \sum_{m=1}^{M} \, \Phi_{k}(x,z_{m}) \, \int_{D_{m}} v_{m}(y) \, dy + \mathcal{O}\left( \delta \, \left\vert \log(\delta) \right\vert^{h-1} \right) \\ \nonumber 
& = & \sum_{m=1}^{M} \, \Phi_{k}(x,z_{m}) \, C_{m}^{\star} \, Q_{m} + \mathcal{O}\left( \delta \, \left\vert \log(\delta) \right\vert^{h-1}  \right) \\
& = & \langle \, \Phi_{k}(x,z) \, \, C^{\star}  ; \tilde{Q} \rangle + \langle \, \Phi_{k}(x,z) \, \, C^{\star} ; \, (Q - \tilde{Q}) \rangle + \mathcal{O}\left( \delta \, \left\vert \log(\delta) \right\vert^{h-1}  \right).
\end{eqnarray}
We need to estimate the second term on the right hand side of the previous equation. To do this, we set $J_{2}:= \langle \, \Phi_{k}(x,z) \, \, C^{\star} \, ; \, (Q - \tilde{Q}) \rangle$. Then,
\begin{equation*}
\vert J_{2} \vert  \leq  \Vert \, \Phi_{k}(x,z)  \, C^{\star} \, \Vert \, \Vert Q - \tilde{Q} \Vert  \overset{(\ref{DiffQ-Q})}{\leq}  \Vert \, \Phi_{k}(x,z) \, \, C^{\star} \, \Vert \; \delta^{1-t} \; \vert \log(\delta) \vert^{h-1}.
\end{equation*}
Again, we suppose that the point $x$ is away  from the domain $D$, which implies the smoothness of the functions $\, \Phi_{k}(x,z_{m})$,\,$ m=1;\cdots;M$, then 
\begin{equation*}
\left\vert J_{2} \right\vert  \lesssim  \left[ \sum_{m=1}^{M}  \left\vert C_{m}^{\star} \, \right\vert^{2} \right]^{\frac{1}{2}} \; \delta^{1-t} \; \vert \log(\delta) \vert^{h-1}  \lesssim  \, \vert \log(\delta) \vert^{h-1} \; \delta^{1-t} \; \vert \log(\delta) \vert^{h-1}.
\end{equation*}
Then, 
\begin{equation}
J_{2} = \mathcal{O}\left(\delta^{1-t} \; \vert \log(\delta) \vert^{2(h-1)} \right).
\end{equation}
Finally, $(\ref{SK})$ becomes, 
\begin{equation}\label{T+R1+R2}
u^{s}(x,\theta) =  \langle \, \Phi_{k}(x,z) \, \, C^{\star} \, ; \tilde{Q} \rangle + \mathcal{O}\left( \delta^{1-t} \, \vert \log(\delta) \vert^{2 \, (h-1)} \right)+ \mathcal{O}\left( \delta \, \vert \log(\delta) \vert^{(h-1)} \right).
\end{equation}
This ends the proof of Theorem \ref{2D-case}. 
\vspace{3mm}
\newline
Now, we define the truncated series that we denote by $\tilde{Q}^{N}, N \in \mathbb{N},$ by $\tilde{Q}^{N} := \underset{n=0}{\overset{N}{\sum}} B_{k}^{n} \cdot U$ and its associated scattered field 
\begin{equation}\label{31M}
u^{s,N}(x,\theta) :=  \langle \, \Phi_{k}(x,z)  \, C^{\star} \, ; \tilde{Q}^{N} \rangle.
\end{equation}
Then, by subtracting $(\ref{31M})$ from $(\ref{T+R1+R2})$, we obtain  
\begin{eqnarray*}
u^{s}(x,\theta) - u^{s,N}(x,\theta) &=&  \langle \, \Phi_{k}(x,z) \, \, C^{\star} \, ; \left( \tilde{Q}  - \tilde{Q}^{N} \right) \rangle \\ &+& \mathcal{O}\left( \delta^{1-t} \, \vert \log(\delta) \vert^{2 \, (h-1)} \right)+ \mathcal{O}\left( \delta \, \vert \log(\delta) \vert^{(h-1)} \right) \\
&=&  \langle \, \Phi_{k}(x,z) \, \, C^{\star} \, ; \left(\sum_{n \geq N+1} B_{k}^{n} \right) \cdot U \rangle \\ &+& \mathcal{O}\left( \delta^{1-t} \, \vert \log(\delta) \vert^{2 \, (h-1)} \right)+ \mathcal{O}\left( \delta \, \vert \log(\delta) \vert^{(h-1)} \right).
\end{eqnarray*}
After taking the modulus, in both sides of the previous equation, we obtain  
\begin{eqnarray*}
\left\vert u^{s}(x,\theta) - u^{s,N}(x,\theta) \right\vert & \leq & \left\vert  \, \Phi_{k}(x,z) \, \right\vert \, \left\vert C^{\star} \right\vert \;  \sum_{n \geq N+1} \left\vert B_{k} \right\vert^{n} \, \left\vert U \right\vert \\ &+& \mathcal{O}\left( \delta^{1-t} \, \vert \log(\delta) \vert^{2 \, (h-1)} \right)+ \mathcal{O}\left( \delta \, \vert \log(\delta) \vert^{(h-1)} \right).
\end{eqnarray*}
Because, $\left\vert U \right\vert \sim 1$ and $C^{\star} \sim C \sim \left\vert \log(\delta) \right\vert^{h-1}$, we reduce the previous inequality to 
\begin{eqnarray*}
\left\vert u^{s}(x,\theta) - u^{s,N}(x,\theta) \right\vert & \lesssim & \left\vert \log(\delta) \right\vert^{h-1} \, \underset{z_{j}}{\min} \left\vert  \Phi_{k} (x,z_{j})\right\vert  \;  \sum_{n \geq N+1} \left\vert B_{k} \right\vert^{n}  \\ &+& \mathcal{O}\left( \delta^{1-t} \, \vert \log(\delta) \vert^{2 \, (h-1)} \right)+ \mathcal{O}\left( \delta \, \vert \log(\delta) \vert^{(h-1)} \right).
\end{eqnarray*}
As $\underset{z_{j}}{\min} \left\vert  \Phi_{k} (x,z_{j})\right\vert \sim 1$, we obtain
\begin{eqnarray*}
\left\vert u^{s}(x,\theta) - u^{s,N}(x,\theta) \right\vert & \lesssim & \left\vert \log(\delta) \right\vert^{h-1}  \, \sum_{n \geq N+1} \left\vert B_{k} \right\vert^{n}  \\ &+& \mathcal{O}\left( \delta^{1-t} \, \vert \log(\delta) \vert^{2 \, (h-1)} \right)+ \mathcal{O}\left( \delta \, \vert \log(\delta) \vert^{(h-1)} \right).
\end{eqnarray*}
From the expression of the matrix $B_{k}$, see $(\ref{defB})$, we can prove that \\ $\left\vert B_{k} \right\vert = \mathcal{O}\left(\left\vert \log(\delta) \right\vert^{h-1} \, \left\vert \log\left( \frac{1}{d} \right) \right\vert \right)$ and using the fact that $d \sim e^{- \left\vert \log(\delta) \right\vert^{t}}$, we deduce that $\left\vert B_{k} \right\vert = \mathcal{O}\left(\left\vert \log(\delta) \right\vert^{t+h-1}  \right)$. Hence, 
\begin{eqnarray*}
\left\vert u^{s}(x,\theta) - u^{s,N}(x,\theta) \right\vert & \lesssim & \left\vert \log(\delta) \right\vert^{h-1}  \, \sum_{n \geq N+1} \left\vert \log(\delta) \right\vert^{n (t+h-1)}  \\ &+& \mathcal{O}\left( \delta^{1-t} \, \vert \log(\delta) \vert^{2 \, (h-1)} \right)+ \mathcal{O}\left( \delta \, \vert \log(\delta) \vert^{(h-1)} \right).
\end{eqnarray*}
Finally, 
\begin{enumerate}
\item[$\ast)$] under the condition $1-t-h > 0$,  
\begin{equation*}
\left\vert u^{s}(x,\theta) - u^{s,N}(x,\theta) \right\vert = \mathcal{O}\left(\left\vert \log(\delta) \right\vert^{(h-1) + (N+1) (t+h-1)} \right).
\end{equation*}
This proves $(\ref{intro-N-interactions-proof-2D})$.
\item[$\ast)$] under the condition $1-t-h = 0$, the equation $(\ref{T+R1+R2})$ becomes
\begin{equation*}
u^{s}(x,\theta) =  \langle \, \Phi_{k}(x,z) \, \, C^{\star} \, ; \tilde{Q} \rangle + \mathcal{O}\left( \delta^{1-t} \, \vert \log(\delta) \vert^{-2 \, t} \right),
\end{equation*}
and we set $u^{s, \infty}(x,\theta) :=  \langle \, \Phi_{k}(x,z) \, \, C^{\star} \, ; \tilde{Q} \rangle$, then:
\begin{equation*}
u^{s}(x,\theta) - u^{s,\infty}(x,\theta) =   \mathcal{O}\left( \delta^{1-t} \, \vert \log(\delta) \vert^{-2 \, t} \right),
\end{equation*}
this proves $(\ref{intro-infty-interactions-proof-2D})$. Form the definition of $u^{s, \infty}(\cdot, \cdot)$ we have 
\begin{eqnarray*}
u^{s, \infty}(x,\theta) &=&  \langle \, \Phi_{k}(x,z) \, \, C^{\star} \, ; \tilde{Q} \rangle \\ &=& \sum_{n \geq 0} \langle \, \Phi_{k}(x,z) \, \, C^{\star} \, ; B_{k}^{n} \cdot U \rangle \\ & \sim & \left\vert \log(\delta) \right\vert^{h-1} \sim \left\vert \log(\delta) \right\vert^{-t} >> \delta^{1-t} \, \vert \log(\delta) \vert^{-2 \, t}.
\end{eqnarray*}
\end{enumerate}
In addition, 
\begin{eqnarray*}
u^{s,N}(x,\theta) - u^{s,N-1}(x,\theta) &=& \langle \Phi_{k}(x,z) C^{\star} ; \left(\tilde{Q}^{N} - \tilde{Q}^{N-1}\right) \rangle \\
\left\vert u^{s,N}(x,\theta) - u^{s,N-1}(x,\theta) \right\vert & \lesssim & \left\vert C^{\star} \right\vert \; \left\vert \tilde{Q}^{N} - \tilde{Q}^{N-1}\right\vert  \lesssim  \left\vert C \right\vert \; \left\vert B_{k} \right\vert^{N}.
\end{eqnarray*}
Using the estimation of $C \sim \left\vert \log(\delta) \right\vert^{h-1}$ and $\left\vert B_{k} \right\vert  = \mathcal{O}\left(\left\vert \log(\delta) \right\vert^{t+h-1}  \right)$ we obtain: 
\begin{equation*}
\left\vert u^{s,N}(x,\theta) - u^{s,N-1}(x,\theta) \right\vert = \mathcal{O}\left(\left\vert \log(\delta) \right\vert^{(h-1)-N(1-h-t)} \right).
\end{equation*}
This concludes the proof of Corollary \ref{coro15}.


\section{Appendix}\label{appendixproposition}
\subsection{Estimation of the scattering coefficient $C_{m}$ and $\left\Vert w \right\Vert_{\mathbb{L}^{2}(D_{m})}$}
\
\\
We recall that: 
\begin{equation*}
C_{m} := \int_{D_{m}} w(x) \, dx,
\end{equation*}
where $w$ is solution, in the domain $D$, of
\begin{equation}\label{EquaW}
\left[ A^{0} - \left( \lambda_{n_{0}} + \mathcal{O}\left( \delta^{2+h} \right) \right)  \right] w = 1 . 
\end{equation}
We can write $C_{m}$, by spectral decomposition of $w$, as:
\begin{equation}\label{wserie}
C_{m} := \int_{D_{m}} w(x) \, dx = \sum_{n=1}^{+\infty} \langle w; e_{n} \rangle_{\mathbb{L}^{2}(D_{m})} \; \int_{D_{m}} e_{n}(x) \; dx.
\end{equation} 
In addition, we have
\begin{eqnarray*}
\int_{D_{m}} e_{n}(x) \, dx  =  \int_{D_{m}} 1 \; e_{n}(x) \, dx & \overset{(\ref{EquaW})}{=} & \int_{D_{m}} \left[ A^{0} - \left( \lambda_{n_{0}} + \mathcal{O}\left( \delta^{2+h} \right) \right)  \right](w)(x) \; e_{n}(x) \; dx \\ 
& = & \langle \left[ A^{0} - \left( \lambda_{n_{0}} + \mathcal{O}\left(\delta^{2+h}\right) \right)  \right] w ; e_{n} \rangle_{\mathbb{L}^{2}(D_{m})} \\ 
& = & \langle A^{0} w ; e_{n} \rangle_{\mathbb{L}^{2}(D_{m})} -  \left( \lambda_{n_{0}} + \mathcal{O}\left(\delta^{2+h}\right) \right) \,  \langle  w ; e_{n} \rangle_{\mathbb{L}^{2}(D_{m})}.
\end{eqnarray*}
Using the fact that $A^{0}$ is a self adjoint operator with $\left(\lambda_{n}; e_{n} \right)$ as an eigen-system we derive, from the previous equation, the coming relation linking the Fourier coefficient of $w$ with the Fourier coefficient of function unity $1$:
\begin{equation}\label{coeffw}
\langle w ; e_{n} \rangle_{\mathbb{L}^{2}(D_{m})} = \frac{\langle 1 ; e_{n} \rangle_{\mathbb{L}^{2}(D_{m})}}{\left( \lambda_{n} - \left( \lambda_{n_{0}} + \mathcal{O}\left( \delta^{2+h} \right)\right)\right)}.
\end{equation}
Then, by combining $(\ref{wserie})$ with $(\ref{coeffw})$, we obtain 
\begin{equation}\label{RMSeries}
C_{m} = \sum_{n=1}^{+\infty} \frac{\left( \langle 1 ; e_{n} \rangle_{\mathbb{L}^{2}(D_{m})} \right)^{2}}{\left( \lambda_{n} - \left( \lambda_{n_{0}} + \mathcal{O}\left( \delta^{2+h} \right)\right)\right)}.
\end{equation} 
To see the behaviour of $\langle 1 ; e_{n} \rangle_{\mathbb{L}^{2}(D_{m})}$, with respect to the radius of $D_{m}$, i.e. $\delta$, we start by scaling: 
\begin{equation*}
\int_{D_{m}} e_{n}(x) \; dx  = \delta^{3} \; \int_{B} e_{n}(z_{m} + \delta \, \xi) \; d\xi = \delta^{3} \; \int_{B_{m}} \tilde{e}_{n}(\xi) \; d\xi.
\end{equation*}
We define the following orthonormal basis: 
\begin{equation*}
\overline{e}_{n} = \frac{\tilde{e}_{n}}{\Vert \tilde{e}_{n} \Vert_{\mathbb{L}^{2}(B)}},
\end{equation*}
where, in effortless manner, we can prove that
$\Vert \tilde{e}_{n} \Vert_{\mathbb{L}^{2}(B)} = \delta^{-\frac{3}{2}}$.
Synthesizing what precedes we deduce: 
\begin{equation}\label{scaleeigfct}
\int_{D_{m}} e_{n}(x) \; dx  = \delta^{\frac{3}{2}} \; \int_{B} \overline{e}_{n}(\xi) \; d\xi.
\end{equation}
Using the derived scales, the equation $(
\ref{RMSeries})$ becomes  
\begin{equation*}
C_{m} = \delta^{3}   \sum_{n \geq 1} \frac{\left( \langle 1; \overline{e}_{n} \rangle_{\mathbb{L}^{2}(B)} \right)^{2}}{\lambda_{n} - (\lambda_{n_{0}} + \mathcal{O}(\delta^{2+h}))},
\end{equation*}
and, by scaling the eigenvalues $\lambda_{n} = \delta^{2} \, \tilde{\lambda}_{n}$ and $\lambda_{n_{0}} = \delta^{2} \, \tilde{\lambda}_{n_{0}}$, we obtain 
\begin{equation}\label{intw}
C_{m}  =  \delta \, \sum_{n \geq 1} \frac{\left( \langle 1; \overline{e}_{n} \rangle_{\mathbb{L}^{2}(B)} \right)^{2}}{\tilde{\lambda}_{n} - (\tilde{\lambda}_{n_{0}} + \mathcal{O}(\delta^{h}))}  = \delta^{1-h} \left( \langle 1 ; \overline{e}_{n_{0}} \rangle_{\mathbb{L}^{2}(B)} \right)^{2} + \delta \sum_{n \neq n_{0}} \frac{\left( \langle 1; \overline{e}_{n} \rangle_{\mathbb{L}^{2}(B)} \right)^{2}}{\tilde{\lambda}_{n} - \left( \tilde{\lambda}_{n_{0}} + \mathcal{O}\left( \delta^{h} \right)\right)}. 
\end{equation}
The preceding formula is consistent if
\begin{equation*}
\sum_{n} \left( \langle 1; \overline{e}_{n} \rangle_{\mathbb{L}^{2}(B)} \right)^{2} \; < \infty.
\end{equation*}
Now, we need to show the convergence of the previous series. For this, let $\beta$ such that
\begin{equation*}
\beta - \underset{n}{\max} \,\, \lambda_{n} >> 1 
\end{equation*}
and $w_{\beta}$ solution of 
\begin{equation}\label{wbeta}
\left( A^{0} - \beta \right) \, w_{\beta} = 1. 
\end{equation}
We have: 
\begin{equation*}
\left( \langle 1; \overline{e}_{n} \rangle_{\mathbb{L}^{2}(B)} \right)^{2} = \left( \beta - \lambda_{n} \right) \; \frac{\left( \langle 1; \overline{e}_{n} \rangle_{\mathbb{L}^{2}(B)} \right)^{2}}{\left( \beta - \lambda_{n} \right)} \leq  \beta \; \frac{\left( \langle 1; \overline{e}_{n} \rangle_{\mathbb{L}^{2}(B)} \right)^{2}}{\left( \beta - \lambda_{n} \right)}.
\end{equation*}
Hence, after taking the series with respect to $n$ on both sides, we get 
\begin{equation*}
\sum_{n} \left( \langle 1; \overline{e}_{n} \rangle_{\mathbb{L}^{2}(B)} \right)^{2} \, \leq  \, \beta \; \sum_{n} \frac{\left( \langle 1; \overline{e}_{n} \rangle_{\mathbb{L}^{2}(B)} \right)^{2}}{\left( \beta - \lambda_{n} \right)}.
\end{equation*}
In addition, using  $(\ref{wbeta})$, we can derive that 
\begin{equation*}
\sum_{n} \frac{\left( \langle 1; \overline{e}_{n} \rangle_{\mathbb{L}^{2}(B)} \right)^{2}}{\left( \beta - \lambda_{n} \right)} = - \int_{D} w_{\beta}(x) \; dx < \infty.
\end{equation*}
Which proves that $(\ref{intw})$ is well defined. Finally, we deduce that 
\begin{equation}\label{C=intW}
C_{m} := \int_{D_{m}} w(x) \, dx = \mathcal{O}\left( \delta^{1-h} \right). 
\end{equation}
To estimate the $\mathbb{L}^{2}-$norm of $w$, we start by writing its series expansion given by 
\begin{equation*}
\left\Vert w \right\Vert^{2}_{\mathbb{L}^{2}(D_{m})} = \sum_{n=1}^{+\infty} \left\vert \langle w; e_{n} \rangle_{\mathbb{L}^{2}(D_{m})} \right\vert^{2} \overset{(\ref{coeffw})}{=} \sum_{n} \frac{\left\vert \langle 1 ; e_{n} \rangle_{\mathbb{L}^{2}(D_{m})} \right\vert^{2}}{\left\vert \left( \lambda_{n} - \left( \lambda_{n_{0}} + \mathcal{O}\left( \delta^{2+h} \right)\right)\right) \right\vert^{2}}.
\end{equation*} 
We use the scale of the eigenvalues $\lambda_{n}$, which is of order $\delta^{2}$, and the scale of the eigenfunction integral, given by $(\ref{scaleeigfct})$, to obtain: 
\begin{equation*}
\left\Vert w \right\Vert^{2}_{\mathbb{L}^{2}(D_{m})} = \delta^{-1} \; \sum_{n} \frac{\left\vert \langle 1 ; \overline{e}_{n} \rangle_{\mathbb{L}^{2}(B)} \right\vert^{2}}{\left\vert \left( \tilde{\lambda}_{n} - \left( \tilde{\lambda}_{n_{0}} + \mathcal{O}\left( \delta^{h} \right)\right)\right) \right\vert^{2}} = \mathcal{O}\left( \delta^{-1-2h} \right).
\end{equation*} 
Finally, we deduce that: 
\begin{equation}\label{estimationnormw}
\left\Vert w \right\Vert_{\mathbb{L}^{2}(D_{m})} =  \mathcal{O}\left( \delta^{-\frac{1}{2}-h} \right).
\end{equation}
\subsection{Proof of Proposition $\ref{aprioriestimateLemma}$}\label{AppendixSubSecII}
\
\newline
We start by setting the L.S.E, in presence of several particles, given by
\begin{eqnarray*}
v_{i}(x) &-& k^{2} \, \frac{1}{a_{0}} \, \int_{D_{i}} \Phi_{k}(x,y) \, (b - b_{0}(y)) \, v_{i}(y) \, dy \\ &-& k^{2} \, \frac{1}{a_{0}} \, \sum_{m \neq i} \int_{D_{m}} \Phi_{k}(x,y) \, (b - b_{0}(y)) \, v_{m}(y) \, dy = u(x).
\end{eqnarray*}
By expanding the function $b_{0}(\cdot)$, near the centers $z_{i}, i=1,\cdots,M$, and rearranging the equation we obtain 
\begin{eqnarray*}
v_{i}(x) &-& k^{2} \, \mu_{0} \, (b - b_{0}(z_{i})) \, \int_{D_{i}} \Phi_{0}(x,y) \,  v_{i}(y) \, dy = u(x) \\ &-& k^{2} \, \frac{1}{a_{0}} \, \int_{D_{i}} \Phi_{0}(x,y) \int_{0}^{1}  \nabla b_{0}(z_{i}+t(y-z_{i})) \cdot (y-z_{i}) \, dt \, v_{i}(y) \, dy \\ &+& k^{2} \, \frac{1}{a_{0}} \, \int_{D_{i}} \left( \Phi_{k}-\Phi_{0} \right)(x,y) \, (b - b_{0}(y)) \, v_{i}(y) \, dy \\ &+& k^{2} \,  \frac{1}{a_{0}} \, \sum_{m \neq i} \int_{D_{m}} \Phi_{k}(x,y) \, (b - b_{0}(y)) \, v_{m}(y) \, dy.
\end{eqnarray*}
After scaling the last equation, we get 
\begin{eqnarray*}
&& \left( I \, - \, k^{2} \, \mu_{0} \, (b - b_{0}(z_{i})) \, \delta^{2} \, A^{0} \right) \tilde{v}_{i}(\eta) 
 =  \tilde{u}_{i}(\eta) \\ 
& - & k^{2} \, \frac{1}{a_{0}} \, \delta^{3} \, \int_{B} \Phi_{0}(\eta,\xi) \, \int_{0}^{1} \widetilde{\nabla b}_{0}\left( z_{i} + t \, \delta \, \xi \right) \cdot \xi \, dt \, \tilde{v}_{i}(\xi) \, d\xi \\
& + & k^{2} \, \frac{1}{a_{0}} \, \delta^{3} \, \int_{B} \left( \tilde{\Phi}_{k}-\tilde{\Phi}_{0} \right)(\eta,\xi) \, (b - \tilde{b_{0}}(\xi)) \, \tilde{v}_{i}(\xi) \, d\xi \\ 
& + & k^{2} \, \frac{1}{a_{0}} \, \delta^{3} \, \sum_{m \neq i} \int_{B} \tilde{\Phi}_{k}(\eta,\xi) \, (b - \tilde{b_{0}}(\xi)) \, \tilde{v}_{m}(\xi) \, d\xi,    
\end{eqnarray*}
where we recall that 
\begin{equation*}
A^{0}(f)(x) := \int_{B} \Phi_{0}(x,y) \, f(y) \, dy.
\end{equation*}
Then 
\begin{eqnarray*}
\tilde{v}_{i} & = & \lambda_{n_{0}} \, \mathcal{R}\left( \lambda_{n_{0}}+ \delta^{h};A^{0}\right) \Bigg[ \tilde{u}_{i}(\cdot) \\ 
& - & k^{2} \, \frac{1}{a_{0}} \, \delta^{3} \, \int_{B} \Phi_{0}(\cdot ,\xi) \, \int_{0}^{1} \widetilde{ \nabla b}_{0}\left( z_{i} + t \, \delta \, \xi \right) \cdot \xi \, dt \, \tilde{v}_{i}(\xi) \, d\xi \\
& + & k^{2} \, \frac{1}{a_{0}} \, \delta^{3} \, \int_{B} \left( \tilde{\Phi}_{k}-\tilde{\Phi}_{0} \right)(\cdot ,\xi) \, (b -\tilde{b_{0}}(\xi)) \, \tilde{v}_{i}(\xi) \, d\xi \\ 
& + & k^{2} \, \frac{1}{a_{0}} \, \delta^{3} \, \sum_{m \neq i} \int_{B} \tilde{\Phi}_{k}(\cdot ,\xi) \, (b - \tilde{b_{0}}(\xi)) \, \tilde{v}_{m}(\xi) \, d\xi \Bigg].  
\end{eqnarray*}
Now, by taking the $\Vert \cdot \Vert_{\mathbb{L}^{2}(B)}$ in both sides of the last equality, we deduce 
\begin{eqnarray*}
\Vert \tilde{v}_{i} \Vert_{\mathbb{L}^{2}(B)} & \leq & \delta^{-h} \Vert \tilde{u}_{i} \Vert_{\mathbb{L}^{2}(B)} \\ 
& + &  \delta^{3} \; \delta^{-h}  \left\Vert \int_{B} (\tilde{\Phi}_{k}-\tilde{\Phi}_{0})(\cdot ,\xi) \, (b - \tilde{b_{0}}(\xi)) \, \tilde{v}_{i}(\xi) \, d\xi \right\Vert_{\mathbb{L}^{2}(B)} \\ 
& + &  \delta^{3} \; \delta^{-h} \;  \sum_{m \neq i} \left\Vert \int_{B} \tilde{\Phi}_{k}(\cdot ,\xi) \, (b - \tilde{b_{0}}(\xi)) \, \tilde{v}_{m}(\xi) \, d\xi \right\Vert_{\mathbb{L}^{2}(B)} \\ 
& + &  \delta^{3} \; \delta^{-h} \left\Vert \int_{B} \Phi_{0}(\cdot , \xi) \int_{0}^{1} \widetilde{\nabla b}_{0}(z_{i}+t \, \delta \, \xi) \cdot \xi \, dt \, \tilde{v}_{i}(\xi) \, d\xi \right\Vert_{\mathbb{L}^{2}(B)},  
\end{eqnarray*}
hence,  
\begin{equation*}
\left\Vert \tilde{v}_{i} \right\Vert_{\mathbb{L}^{2}(B)}  \leq  \delta^{-h} \, \left\Vert \tilde{u}_{i} \right\Vert_{\mathbb{L}^{2}(B)} +  \delta^{1-h} \,  \left\Vert \tilde{v}_{i} \right\Vert_{\mathbb{L}^{2}(B)} +  \delta^{1-h} \,  \sum_{m \neq i} \left\Vert \tilde{v}_{m} \right\Vert_{\mathbb{L}^{2}(B)}  +  \delta^{3-h} \,  \left\Vert  \tilde{v}_{i} \right\Vert_{\mathbb{L}^{2}(B)},
\end{equation*}
which can be reduced to
\begin{eqnarray*}
\left\Vert \tilde{v}_{i} \right\Vert_{\mathbb{L}^{2}(B)} & \leq &  \delta^{-h} \, \left\Vert \tilde{u}_{i} \right\Vert_{\mathbb{L}^{2}(B)} +  \delta^{1-h} \,  \left\Vert \tilde{v}_{i} \right\Vert_{\mathbb{L}^{2}(B)} +  \delta^{1-h} \,  \sum_{m \neq i} \left\Vert \tilde{v}_{m} \right\Vert_{\mathbb{L}^{2}(B)} \\
\left\Vert \tilde{v}_{i} \right\Vert_{\mathbb{L}^{2}(B)} & \leq & \left( 1-\delta^{1-h} \right)^{-1} \; \delta^{-h} \, \left\Vert \tilde{u}_{i} \right\Vert_{\mathbb{L}^{2}(B)} + \left( 1 - \delta^{1-h} \right)^{-1} \, \delta^{1-h} \;  \sum_{m \neq i} \left\Vert \tilde{v}_{m} \right\Vert_{\mathbb{L}^{2}(B)} \\ 
& \lesssim &  \delta^{-h} \, \left\Vert \tilde{u}_{i} \right\Vert_{\mathbb{L}^{2}(B)} + \delta^{1-h} \;  \sum_{m \neq i} \left\Vert \tilde{v}_{m} \right\Vert_{\mathbb{L}^{2}(B)}.
\end{eqnarray*}
Scaling back we obtain 
\begin{equation*}
\left\Vert v_{i} \right\Vert_{\mathbb{L}^{2}(D_{i})} \leq \delta^{-h} \left\Vert u_{i} \right\Vert_{\mathbb{L}^{2}(D_{i})} +  \delta^{1-h} \;  \sum_{m \neq i} \left\Vert v_{m} \right\Vert_{\mathbb{L}^{2}(D_{m})},
\end{equation*}
and 
\begin{equation*}
\left\Vert v_{i} \right\Vert_{\mathbb{L}^{2}(D_{i})} \leq \delta^{-h} \left\Vert u_{i} \right\Vert_{\mathbb{L}^{2}(D_{i})} +  \delta^{1-h} \;  \sum_{m=1}^{M} \left\Vert v_{m} \right\Vert_{\mathbb{L}^{2}(D_{m})}.
\end{equation*}
By taking the modulus squared in both sides we obtain 
\begin{equation*}
\left\Vert v_{i} \right\Vert^{2}_{\mathbb{L}^{2}(D_{i})} \lesssim \delta^{-2h} \, \left\Vert u_{i} \right\Vert^{2}_{\mathbb{L}^{2}(D_{i})} +  \delta^{2(1-h)} \; \sum_{m=1}^{M} \left\Vert v_{m} \right\Vert^{2}_{\mathbb{L}^{2}(D_{m})}.
\end{equation*}
Summing up with respect to the index $i$, we get  
\begin{equation}\label{CCM}
\sum_{i=1}^{M} \, \left\Vert v_{i} \right\Vert^{2}_{\mathbb{L}^{2}(D_{i})}  \lesssim  \delta^{-2h} \,\sum_{i=1}^{M} \,   \left\Vert u_{i} \right\Vert^{2}_{\mathbb{L}^{2}(D_{i})} +  \delta^{2(1-h)} \;  \sum_{m=1}^{M} \left\Vert v_{m} \right\Vert^{2}_{\mathbb{L}^{2}(D_{m})}.
\end{equation}
We deduce, under the condition $h<1$, that:
\begin{eqnarray*}
\sum_{i=1}^{M} \, \left\Vert v_{i} \right\Vert^{2}_{\mathbb{L}^{2}(D_{i})}  & \leq & \left( 1 - \delta^{2(1-h)} \right)^{-1} \;\; \delta^{-2h} \,\sum_{i=1}^{M} \,   \left\Vert u_{i} \right\Vert^{2}_{\mathbb{L}^{2}(D_{i})}.
\end{eqnarray*}
This allows us, under the condition $h <1$, to deduce that
\begin{equation}
\left\Vert v \right\Vert_{\mathbb{L}^{2}(D)} \;  \leq   \delta^{-h} \; \left\Vert u \right\Vert_{\mathbb{L}^{2}(D)}.
\end{equation}

\end{document}